\documentclass[reqno]{amsart}
\usepackage{amssymb}
\usepackage{amsmath}
\usepackage{amsthm}
\usepackage{mathtools}
\usepackage[svgnames]{xcolor}
\usepackage[pdftex]{hyperref}
\usepackage{enumitem}
\usepackage{graphicx}
\usepackage{algorithm}
\usepackage{algorithmic}
\usepackage{subcaption}

\hypersetup{
  pdftitle={},
  pdfauthor={Stefan Le Coz <slecoz@math.univ-toulouse.fr>},
  pdfsubject={},
  pdfkeywords={}, 
  colorlinks=true,
  linkcolor=DarkBlue  ,          
  citecolor=DarkRed,        
  filecolor=DarkMagenta,      
  urlcolor=DarkGreen,           
}

\newtheorem{theorem}{Theorem}[section]
\newtheorem{proposition}[theorem]{Proposition}
\newtheorem{lemma}[theorem]{Lemma}

\theoremstyle{remark}
\newtheorem{remark}[theorem]{Remark}

\newtheorem*{ackn}{Acknowledgments}

\DeclarePairedDelimiter{\norm}{\lVert}{\rVert}

\newcommand{\dual}[2]{\left\langle #1,#2 \right\rangle}

\newcommand{\eps}{\varepsilon}

\newcommand{\R}{\mathbb{R}}

\DeclareMathOperator{\sign}{sign}

\renewcommand{\leq}{\leqslant}
\renewcommand{\geq}{\geqslant}

\DeclareMathAlphabet{\mathpzc}{OT1}{pzc}{m}{it}

\usepackage{color}

\begin{document}

\title[]{Computation of excited states\\ for the nonlinear Schrödinger equation: \\numerical and theoretical analysis}

\author[C.~Besse]{Christophe Besse}
\author[R.~Duboscq]{Romain Duboscq}
\author[S.~Le Coz]{Stefan Le Coz}

\thanks{This work was supported by the ANR LabEx CIMI (grant ANR-11-LABX-0040) within the French State Programme "Investissements d'Avenir" and the ANR project NQG ANR-23-CE40-0005.}
\address[Christophe Besse]{Institut de Math\'ematiques de Toulouse ; UMR5219,
  \newline\indent
  Universit\'e de Toulouse ; CNRS,
  \newline\indent
  UPS IMT, F-31062 Toulouse Cedex 9,
  \newline\indent
  France}
\email[Christophe Besse]{christophe.besse@math.univ-toulouse.fr}
\address[Romain Duboscq]{Institut de Math\'ematiques de Toulouse ; UMR5219,
  \newline\indent
  Universit\'e de Toulouse ; CNRS,
  \newline\indent
  UPS IMT, F-31062 Toulouse Cedex 9,
  \newline\indent
  France}
\email[Romain Duboscq]{romain.duboscq@math.univ-toulouse.fr}

\address[Stefan Le Coz]{Institut de Math\'ematiques de Toulouse ; UMR5219,
  \newline\indent
  Universit\'e de Toulouse ; CNRS,
  \newline\indent
  UPS IMT, F-31062 Toulouse Cedex 9,
  \newline\indent
  France}
\email[Stefan Le Coz]{slecoz@math.univ-toulouse.fr}


\date{\today}

\begin{abstract}
Our goal is to compute excited states for the nonlinear Schr\"odinger equation in the radial setting. We introduce a new technique based on the Nehari manifold approach and give a comparison with the classical shooting method. We observe that the Nehari method allows to accurately compute excited states on large domains but is relatively slow compared to the shooting method. 
\end{abstract}

\maketitle

\setcounter{tocdepth}{1}

\section{Introduction}

We consider the nonlinear Schr\"odinger equation
\begin{equation}\label{eq:nls}
 iu_t+\Delta u+f(u)=0
\end{equation}
where $u:\R\times\R^d\to\mathbb C$ and $f\in\mathcal{C}^1(\R,\R)$ is an odd function extended to $\mathbb C$ by setting $f(z)=f(|z|)z/|z|$ for all $z\in\mathbb C\setminus\{0\}$. Equation \eqref{eq:nls} arises in various physical contexts, for example in nonlinear optics or in the modelling of Bose-Einstein condensates. For physical applications as well as for its numerous interesting mathematical properties, \eqref{eq:nls} has been the subject of an intensive research over the past forty years. We refer for example to the books of Cazenave \cite{Ca03}, Fibich \cite{Fi15} and Sulem and Sulem \cite{SuSu99} for an overview of the known properties of \eqref{eq:nls} and references.

In this paper, we focus on special solutions of \eqref{eq:nls}, the so-called \emph{standing waves}. They are solutions of the form $e^{i\omega t}\phi(x)$ with $\omega>0$ and $\phi$ satisfying
\begin{equation}\label{eq:snls}
\left\{
\begin{array}{l}
-\Delta\phi+\omega\phi-f(\phi)=0,\\
\phi\in H^1(\mathbb R^d) \setminus\{0\}.
\end{array}
\right.
\end{equation}
Among solutions of \eqref{eq:snls}, it is common to distinguish between the \emph{ground states}, or \emph{least energy solutions}, and the other solutions, the \emph{excited states}. A \emph{ground state} is a solution of \eqref{eq:snls} minimizing among all solutions of \eqref{eq:snls} the functional $S$, often called  \emph{action}, defined for $v\in H^1(\mathbb R^d) $ by
\begin{equation}\label{eq:action}
S(v):=\frac{1}{2}\norm{\nabla v}_{L^2}^2+\frac{\omega}{2}\norm{v}_{L^2}^2-\int_{\mathbb R^d} F(v)dx,
\end{equation}
where $F(z):=\int_0^ {|z|}f(s)ds$ for all $z\in\mathbb C$. An \emph{excited state} is a solution of \eqref{eq:snls} which is \emph{not} a ground state. In general, we shall refer to any solution of \eqref{eq:snls} as \emph{bound state}.
Sufficient and almost necessary hypotheses on $f$ to ensure the existence of bound states are known since the fundamental works of Berestycki and Lions \cite{BeLi83-1,BeLi83-2} and Berestycki, Gallou\"et and Kavian \cite{BeGaKa83}. Under these hypotheses, it is proved in \cite{BeGaKa83,BeLi83-1,BeLi83-2} that, except in dimension $d=1$ where all bound states are ground states, there exist ground states and infinitely many excited states.

Note that the terminology \emph{ground state} may be understood in different ways depending on the context. Some authors may call ground state any minimizer of the energy functional under the mass constraint. For  power-type mass-subcritical nonlinear Schr\"odinger equations, this definition will coincide with ours. However, in other settings such as for the power-type mass-supercritical nonlinear Schr\"odinger equations, the two definitions do not coincide any more (there is no ground state in the later sense).

With our definition of ground states, it has long been established for power-type nonlinearities that ground states are positive, radial, and unique (see \cite{GiNiNi79,Kw89}). On the other hand, excited states will necessarily change sign, and may even be complex valued (see e.g. \cite{Li86}). They also need not to be radial, nor to even have any kind of symmetry group (as was shown recently in \cite{AoMuPaWe16}).

There are numerous works devoted to the numerical calculations of ground states. Among many others, we find the seminal work of Bao and Du \cite{BaDu04} devoted to the gradient flow with discrete normalisation which has been used in many settings (including by the authors of the present paper in the context of quantum graphs \cite{BeDuLe22A,BeDuLe22B}). We also mention the work of Choi and McKenna \cite{zbMATH00223162} devoted to the numerical implementation of the Mountain Pass approach, which has also been followed by numerous extensions and improvements (see e.g. \cite{MaRaRuSo21}). The Mountain Pass approach can be modified to compute nodal states, as was done by Costa, Ding and Neuberger \cite{CoDiNe11}, whose approach was later followed by Bonheure, Bouchez, Grumiau, and van Schaftingen \cite{BoBoGrva08}. Not many other works in the literature are devoted to the calculation of excited states, and to our knowledge this paper is the first one to present an approach based on  the Nehari manifold.

Our goal in this article is to develop numerical methods for the computation of excited states. We will also take this opportunity to study numerically some properties of the excited states and establish some conjectures that could be further investigated theoretically.

The two methods that we are considering are the \emph{shooting method} and the \emph{Nehari method}. The \emph{shooting method} is a classical method for the computation of solutions of boundary value problems. It consists in transforming the $d$-dimensional partial differential equation \eqref{eq:snls} into an ordinary differential equation by considering real-valued radial solutions. The boundary value problem for the ordinary differential equation is then converted into an initial value problem, which can be easily solved using a  standard scheme such as the Runge-Kutta 4th order method. In the present case, since we are working with an elliptic problem, we are led to consider initial conditions on the solution and its first derivative. The first derivative is necessarily set to $0$ since its originating from a smooth radial function. We are thus left with the initial value of the solution which will be used as a parameter to be chosen in order to recover the boundary condition at infinity. The method is described in Section \ref{sec:shooting}.

The idea of the \emph{Nehari method} originates from the variational characterization of bound states as minimizers of the action functional under constraints build upon the Nehari functional. In the case of the ground state, we simply minimize the action among functions for which the Nehari functional vanishes. To obtain excited states, more elaborate constructions are required. For example, one may minimize the action among real-valued functions having non-trivial positive and negative parts both satisfying the Nehari constraint.  We establish the existence of a solution (called \emph{least energy nodal minimizer}) to this variational problem in a radial setting in Theorem \ref{thm:least-energy-nodal}.
We also establish the existence of radial excited states vanishing on a given number of nodal regions in Theorem \ref{thm:least-energy-multinodal}. 
We refer to Section \ref{sec:some-theor-elem} for more details on the theoretical background. The numerical approach is based on  a projected gradient method which consists in one step of gradient flow for the action followed by a projection on the chosen space of constraints. The method is described in Section \ref{sec:nehari}.  While being a standard theoretical tool, the Nehari approach has been seldom used in numerical analysis. To our knowledge, this paper is the first to investigate the computation of excited states using a Nehari approach. 

While the shooting method is the go-to method for finding excited states, we identified some limitations of this approach. Indeed, it can only compute radial excited states, while the Nehari method could be extended to non-radial problems (see Remark \ref{rmk:compactness}). Moreover, even in the radial case, the length of the interval on which the solution is computed is limited for  the shooting method due to propagation of the error from the initial condition. This issue is not present for the Nehari approach. On the other hand, the convergence of the Nehari approach is much slower compared to the shooting method, which has a maximal number of iterations for a given precision.
The two methods are discussed in Section \ref{sec:comparison}.

We conclude this paper by some numerical experiments. We investigate numerically the relation between the initial values of the bound states and their total number of nodes. We also study the positions of the nodes and the extremal values between two consecutive nodes. For each case, we provide some guess on the underlying behavior. This material is presented in Section \ref{sec:exp}.

\begin{ackn}
We are grateful to the anonymous reviewers for their careful reading and relevant comments which helped to  improve the quality of the paper. 
\end{ackn}

\section{Theoretical approach for the minimization over the Nehari manifold}
\label{sec:some-theor-elem}

Our goal in this section is to present some theoretical elements around ground states and excited states and the Nehari minimization approach. 
We start by reviewing some well-known facts for the existence and properties of ground states and excited states. 
We also present the classical variational characterization of ground states as minimizers on the Nehari manifold (see Proposition \ref{prop:nehari}).
The rest of the section is then devoted to the statement and proof of our two main theoretical results. First, we establish a characterization of \emph{the first nodal radial excited state} as a minimizer on the Nehari nodal set (see Theorem \ref{thm:least-energy-nodal}). Second, we obtain a series of nodal radial excited states having a prescribed number of nodes $N$, which we call \emph{lower radial $N$-nodes excited state}, by minimization over multiple Nehari nodal sets (see Theorem \ref{thm:least-energy-multinodal}).
While the approach to obtain Theorems \ref{thm:least-energy-nodal} and \ref{thm:least-energy-multinodal} borrows elements from the existing literature (see e.g. \cite{SzWe10}), the results themselves seem to be new.

We consider \eqref{eq:snls} with solutions belonging to $H^1(\R^d,\R)$. Note that even if the problem originates from solutions of Schr\"odinger equation which are complex-valued, we restrict the study of stationary states to \emph{real valued} solutions. 
A typical example for $f$ is the power-type nonlinearity
$f(u)=|u|^{p-1}u$, $1<p<2^*-1$, where $2^*$ is the critical Sobolev
exponent, i.e. $2^*=\frac{2d}{d-2}$ if $d\geq 3$, $2^*=\infty$ if
$d=1,2$. More generally, we assume that $f:\R\to\R$ verifies the following
hypotheses (which are not optimal, but sufficient for our purpose). 
\begin{enumerate}[label=\textbf{(H\arabic*)}]

\item \label{item:h1} (regularity) The function $f$ is continuous and odd.
\item \label{item:h2} (subcriticality) There exists $1<p<2^*-1$ such that for large $s$,  $|f(s)|\lesssim
  |s|^{p}$.
\item \label{item:h3} (superlinearity) At $0$, $\lim_{s\to0}\frac{f(s)}{s}=0$.
\item \label{item:h4} (focusing) There exists $\xi_0>0$ such that $F(\xi_0)=\int_0^{\xi_0}f(s)ds>\frac{\xi_0^2}{2}$.
\end{enumerate}
Under \ref{item:h1}-\ref{item:h4}, it is well known (see \cite{BeGaKa83,BeLi83-1}) that
there exist \emph{ground state}
 solutions, i.e. solutions with minimal action
(see \eqref{eq:action} for the definition of the action)
among all possible solutions to \eqref{eq:snls}. Our definition of 
ground states as minimal action solutions is very common in the
analysis of nonlinear elliptic partial differential equations. However, as already said, the
terminology ground state has several other acceptations in
other contexts. E.g. when working with Schr\"odinger equations modelling
  Bose-Einstein condensation, many authors
  call ground state a minimizer of the energy on fixed mass
  constraint.

Uniqueness of the ground state
holds if $f$ satisfies in addition to \ref{item:h1}-\ref{item:h4} some
complementary requirements, e.g. if $f$ is of
power-type, see \cite{Kw89}. When $d\geq 2$, it was
proved in \cite{BeGaKa83,BeLi83-2} that there exists an infinite sequence of
\emph{excited states}, i.e. solutions to \eqref{eq:snls} whose action is not
minimal (actually, the corresponding sequence of actions tends to
infinity). 
Uniqueness of radial excited states for a prescribed number of nodes has long been an open problem, even in the case of power non-linearities, with only partial results in this direction (see e.g. \cite{CoGaYa09,CoGaYa11}). The conjecture has very recently been established for power non-linearities in \cite{Ta24}.

Recall that the action functional $S:H^1(\R^d)\to\R$ is defined in \eqref{eq:action}. It is a $\mathcal C^1$ functional (see
e.g. \cite{AmMa07}) and
$u$ is a solution of \eqref{eq:snls} if and only if $S'(u)=0$. We
define the \emph{Nehari functional} by
\[
I(u)=\dual{S'(u)}{u}=\norm{\nabla u}_{L^2}^2+\norm{u}_{L^2}^2-\int_{\R^d}f(u)udx.
\]
The \emph{Nehari manifold} is defined by
\[
\mathcal N=\{u\in H^1(\R^d)\setminus\{0\}:I(u)=0\}.
\]
Define the \emph{Nehari level} by
\[
m_{\mathcal N}=\inf\{S(v):v\in\mathcal N\}.
\]
In addition to \ref{item:h1}-\ref{item:h4}, we assume the following.
\begin{enumerate}[label=\textbf{(H\arabic*)}]
  \setcounter{enumi}{4}
\item \label{item:h5} The function $s\to\frac{f(s)}{s}$ is increasing for $s>0$. 
\item \label{item:h6} (Ambrosetti-Rabinowitz superquadraticity
condition) There exists $\theta>2$ such that $\theta F(s)<sf(s)$ for all $s>0$. 
\end{enumerate}
Then under \ref{item:h1}-\ref{item:h6}, the following holds (see e.g. \cite{SzWe10} and
the references cited therein).
\begin{proposition}
\label{prop:nehari}
 For every sequence $(u_n)\in \mathcal N$ such that 
\[
\lim_{n\to\infty}S(u_n)=m_{\mathcal N},
\]
there exist $u_\infty\in\mathcal N$ and $(y_n)\subset \R^d$ such that, up to a subsequence,
\[
\lim_{n\to\infty}\norm{u_n(\cdot-y_n)-u_\infty}_{H^1}=0.
\] 
Moreover, $u_\infty$ is a ground state solution of \eqref{eq:snls}.
\end{proposition}

We now want to construct variational characterizations of excited
states which can be used in numerical approaches. 
Based on Proposition \ref{prop:nehari}, it is natural to try to
generalize the Nehari manifold approach. Several directions of
investigations are possible. The most natural one is probably to
define the \emph{Nehari nodal set} as
\[
\mathcal N_{\mathrm{nod}}=\{u\in H^1(\R^d):I(u^+)=0,I(u^-)=0,u^\pm\neq0\}.
\]
where $u^+=\max(u,0)$ and $u^-=\max(-u,0)$. Define the \emph{Nehari nodal
level} by
\[
m_{\mathcal N_{\mathrm{nod}}}=\inf\{S(v):v\in\mathcal N_{\mathrm{nod}}\}.
\]
\begin{remark}
  An approach based on minimization of the energy on mass
  constraints for the positive and negative parts of the function
  cannot work, as the minimizer that we might obtain would be
  (formally) a
  solution of an equation of the form
\[
E'(u)+\lambda_+M'(u_+)+\lambda_-M'(u_-) = 0,
\] 
with potentially different Lagrange multipliers $\lambda_\pm$. This
issue is avoided with the Nehari approach.  
\end{remark}

\begin{lemma}
\label{lem:not-achieved}
Assume \ref{item:h1}-\ref{item:h6}. The nodal Nehari infimum  level $m_{\mathcal N_{\mathrm{nod}}}$ satisfies
\begin{equation}\label{eq:nehari-levels}
m_{\mathcal N_{\mathrm{nod}}}=2m_{\mathcal N}.
\end{equation}
Moreover, there is no function achieving $m_{\mathcal N_{\mathrm{nod}}}$. 
\end{lemma}

\begin{proof}
 Let $u\in \mathcal
N_{\mathrm{nod}}$. Since $u^+$ and $u^-$ are both in $\mathcal N$ and have disjoint support, we have
\[
S(u)=S(u^+)+S(u^-)\geq 2m_{\mathcal N},
\]
and therefore $m_{\mathcal N_{\mathrm{nod}}}\geq 2m_{\mathcal N}$.
Let $u_\infty$ be a minimizer for $m_{\mathcal N}$. 
We can assume that $u_\infty$ is positive, radial and exponentially decreasing (see e.g. \cite{Ca03,GiNiNi79}). 

Let 
$(y_n)\subset\R^d$ and define
\begin{equation}\label{eq:example}
u_n=u_\infty(\cdot+y_n)-u_\infty(\cdot-y_n). 
\end{equation}
When $|y_n|\to \infty$,
we have
\[
S(u_n)\to 2m_{\mathcal N}.
\]

We assume that for each $n$, the translation parameter $y_n$ is of the form $(y_n^1,0,\dots,0)$. Then $u_n$ is odd with respect to the first component. Moreover, $(u_n)_+$ is supported on $\{x^1\geq0\}$ and $(u_n)_-$ is supported on $\{x^1\leq0\}$. Since $u_n$ and $\nabla u_n$ are exponentially decaying, we have in $H^1(\R^d)$ the strong convergence 
\[
(u_n)_+(\cdot-y_n)\to u_\infty,\quad (u_n)_-(\cdot+y_n)\to u_\infty.
\]
In particular, we have 
\[
\lim_{n\to \infty}I((u_n)_+)=
\lim_{n\to \infty}I((u_n)_-)=I(u_\infty)=0.
\]
Hence there exists $t_n>0$, $t_n\to 1$ such that, for all $n\in\mathbb N$, we have
\[
I(t_n(u_n)_+)=I(t_n(u_n)_-)=0.
\]
Writing $\tilde u_n=t_n u_n$, we have $\tilde u_n\in \mathcal N_{\mathrm{nod}}$ and 
\[
S(\tilde u_n)\to 2m_{\mathcal N}.
\]
This implies that $m_{\mathcal N_{\mathrm{nod}}}\leq 2m_{\mathcal N}$. As we already established the reverse inequality, this proves \eqref{eq:nehari-levels}. 

Unfortunately,
$m_{\mathcal N_{\mathrm{nod}}}$ is not achieved. Indeed, suppose on the contrary
that $u_{\mathrm{nod}}$ realizes the minimum for $m_{\mathcal N_{\mathrm{nod}}}$. Since
$u_{\mathrm{nod}}^\pm\in\mathcal N$ and $m_{\mathcal N_{\mathrm{nod}}}=2m_{\mathcal N}$,
 both $u_{\mathrm{nod}}^+$ and $u_{\mathrm{nod}}^-$ realize the minimum for $\mathcal
 N$ and are ground states of \eqref{eq:snls}. In particular, they are
 both regular, and by the maximum principle, both have to be positive or negative on the whole
 $\R^d$, which is a contradiction. Therefore $m_{\mathcal N_{\mathrm{nod}}}$ is
 not achieved - from \eqref{eq:example}, we can easily guess that this
 is due to a loss of compactness in the minimizing sequences.   
  \end{proof}

\begin{remark}
On the other hand, if the power nonlinearity
  $|u|^{p-1}u$ is replaced by a Choquart/Hartree term
  (e.g. $(|x|^{-1}*|u|^2)u$ in $\R^3$), then it is possible to obtain
  nodal critical points by minimizing $S$ on
  $\mathcal N_+\cap\mathcal N_-$, see \cite{GhVa16}.
\end{remark}

To
 overcome the issue raised by Lemma \ref{lem:not-achieved}, we decide to work in a radial setting. Recall that the restriction of $H^1(\R^d)$ to radial functions is the space denoted $H^1_{\mathrm{rad}}(\R^d)$ defined by
 \[
 H^1_{\mathrm{rad}}(\R^d)=\{u\in H^1(\R^d):\exists v:[0,\infty)\to \R,\forall x\in\R^d, \,u(x)=v(|x|) \}.
 \]
For $u\in  H^1_{\mathrm{rad}}(\R^d)$ and the corresponding $v$, we have
\[
\int_{\R^d}\left(|\nabla u(x)|^2+|u(x)|^2\right)dx=
2\pi\int_0^\infty (|v'(s)|^2+|v(s)|^2)s^{d-1}ds.
\]
In particular, $v\in H^1(\R^+,s^{d-1}ds)$. As a consequence, the function $v$ has a representative which is continuous on $(0,\infty)$ (but may be singular at $0$). In the sequel, we will use the same notation for $u$ and its radial representative, i.e. for $u$ radial we write $u(x)=u(|x|)$.
Recall from Strauss' Lemma \cite{St77} that the injection $H^1_{\mathrm{rad}}(\R^d)\hookrightarrow L^q(\R^d)$, $2<q<2^*$ is compact whenever $d\geq2$. 
 
 \begin{remark}
 \label{rmk:compactness}
 We have made the choice to work in a radial setting as it allows comparison with the results given by the shooting method. The key point to recover nodal minimizers is a compactness property. 
Hence, the results of the present article around the Nehari method for two nodal components could be extended to bounded domains, confining potentials, etc. It is however unclear  how to generalize this approach to recover solutions with more than two nodal domains.
 \end{remark}
 
 Define 
\[
\mathcal N_{\mathrm{nod,rad}}=\{u\in H^1_{\mathrm{rad}}(\R^d):I(u^+)=0,I(u^-)=0,u^\pm\neq0\},
\]
and 
\[
m_{\mathcal N_{\mathrm{nod,rad}}}=\inf\{S(v):v\in\mathcal N_{\mathrm{nod,rad}}\}.
\]
Then the following result gives the existence of a minimizer for
$m_{\mathcal N_{\mathrm{nod,rad}}}$. 

\begin{theorem}
\label{thm:least-energy-nodal}
For every sequence $(u_n)\in \mathcal N_{\mathrm{nod,rad}}$ such that 
\[
\lim_{n\to\infty}S(u_n)=m_{\mathcal N_{\mathrm{nod,rad}}}
\]
there exists $u_\infty\in\mathcal N_{\mathrm{nod,rad}}$ such that, up to a subsequence,
\[
\lim_{n\to\infty}\norm{u_n-u_\infty}_{H^1}=0.
\] 
Moreover, $u_\infty$ is a nodal solution of \eqref{eq:snls} with
exactly two nodal domains. We say
that $u_\infty$ is a \emph{least nodal excited state}.
\end{theorem}

\begin{remark}
  Minimizing on $\mathcal N_{\mathrm{nod,rad}}$ is
  intrinsically more difficult than minimizing on $\mathcal
  N$.
  Indeed, $\mathcal N_{\mathrm{nod,rad}}$ is \emph{not} a manifold, as the
  functionals
  \[
  u\in H^1(\R^d)\to \norm{\nabla u^\pm}_{L^2}^2
  \]
  are \emph{not} $\mathcal C^1$ (see the discussion after Theorem 18
  in \cite{SzWe10}).
\end{remark}

To compute numerically nodal solutions, an approach based on Theorem \ref{thm:least-energy-nodal} would provide only the least radial nodal excited state. As we would also like to compute higher order excited states, we will adopt a slightly different setting.

Fix $N_{\textrm{nodes}}\in\mathbb{N}$  the number of desired nodes. Let $\Omega(\rho,\sigma)\subset\R^d$ be the
annulus of $\R^d$ of
radii $\rho$ and $\sigma$:
\[
\Omega(\rho,\sigma)=\{x\in\R^d:\rho\leq |x|<\sigma\}.
\]
Given $u\in H_{\textrm{rad}}^1(\mathbb R^d)$, denote by $u_{|[\rho,\sigma]}$ the restriction of $u$ to $\Omega(\rho,\sigma)$, i.e.
   \[
u_{|[\rho,\sigma]} (x)=
\begin{cases}
  u(x)&\text{if } x\in \Omega(\rho,\sigma),\\
  0 & \text{otherwise.}
\end{cases}
     \]
Denote the set of functions having at least $N_{\textrm{nodes}}+1$ nodal components by 
  \begin{multline*}
    \mathcal Z_{N_{\textrm{nodes}}}:=\big\{ u \in H_{\textrm{rad}}^1(\mathbb R^d),\exists (\rho_k)_{k=1,\dots, N_{\textrm{nodes}}}\in u^{-1}(0),\\ 0 = \rho_0 < \rho_1<\ldots<\rho_{N_{\textrm{nodes}}+1} = \infty, \,\, u_{|[\rho_k,\rho_{k+1}]}\neq 0 
    ,\text{ for }k=0,\dots,N_{\textrm{nodes}},
     \\
       \sign{u(\rho_k^+)}
       \cdot \sign{u(\rho_k^-)}< 0 
       \text{ for }k=1,\dots, N_{\textrm{nodes}}
       \big\},
     \end{multline*}
    where by $\sign{u(\rho_k^+)}$ we mean the limit of $\sign{u(x)}$ when $x\to\rho_k$, $x>\rho_k$, $u(x)\neq 0$ (i.e. the condition requires $u$ to change sign around $\rho_k$).
  For $u\in \mathcal Z_{N_{\textrm{nodes}}}$, we denote the restriction of $u$ to $\Omega(\rho_k^n,\rho_{k+1}^n )$ by
  \[
u^k:=  u_{|[\rho_k,\rho_{k+1}]}.
    \]
Define
\begin{equation*}
\mathcal{N}_{N_{\textrm{nodes}}} : = \left\{u\in \mathcal Z_{N_{\textrm{nodes}}},\; I(u^k) = 0,\; 0\leq k \leq N_{\textrm{nodes}}  \right\},
\end{equation*}
and
\[
m_{\mathcal{N}_{N_{\textrm{nodes}}}}=\inf\{S(v):v\in\mathcal{N}_{N_{\textrm{nodes}}}\}.
\]
Then the following result gives the existence of a minimizer for
$m_{\mathcal{N}_{N_{\textrm{nodes}}}}$ which corresponds to a higher order excited state.

\begin{theorem}
\label{thm:least-energy-multinodal}
For every sequence $(u_n)\in\mathcal{N}_{N_{\textrm{nodes}}}$ such that 
\[
\lim_{n\to\infty}S(u_n)=m_{\mathcal{N}_{N_{\textrm{nodes}}}}
\]
there exists $u_\infty\in\mathcal{N}_{N_{\textrm{nodes}}}$ such that, up to a subsequence, 
\[
\lim_{n\to\infty}\norm{u_n-u_\infty}_{H^1}=0.
\] 
Moreover, $u_\infty$ is solution of \eqref{eq:snls} with
exactly $N_{\textrm{nodes}}+1$ nodal domains and we say
that $u_\infty$ is a \emph{lower radial $N_{\textrm{nodes}}$-nodes excited state}.
\end{theorem}

\begin{remark}
A related result, consisting in pasting together solutions of the Nehari problem on annuli with Dirichlet boundary conditions and then optimizing over the radii, was obtained by Bartsch and Willem, see \cite{BaWe93}. 
\end{remark}

The rest of this section is devoted to the proof of Theorems \ref{thm:least-energy-nodal} and \ref{thm:least-energy-multinodal}. We start with some preliminary lemmas.

\begin{lemma}
\label{lem:scaling}
The constant $0$ is a local minimum for $S$. 
  Let $u\in H^1(\R^d)\setminus\{0\}$. There exists a unique $s_u\in
  (0,\infty)$ such that $I(s_u u)=0$. Moreover, $S(s_u
  u)=\max_{s\in(0,\infty)}S(su)>0$. If $I(u)<0$, then $s_u<1$, whereas
  if $I(u)>0$, then $s_u>1$ and $S(u)>0$. 
\end{lemma}

\begin{proof}
Let $u\in H^1(\R^d)\setminus\{0\}$ and define $h:(0,\infty)\to\R$ by 
\[
h(s):=S(su)=\frac{s^2}{2}\norm{u}_{H^1}^2-\int_\R F(su)dx.
\]
Since $F$ is differentiable, so is $h$ and we have
\[
h'(s)=s \norm{u}_{H^1}^2-\int_\R f(su)udx.
\]
Remark that $sh'(s)=I(su)$. Due to \ref{item:h5}, the derivative $h'$ can vanish
only once in $(0,\infty)$. Indeed, assume by contradiction that
there exist $0<s_1<s_2$ such that $h'(s_1)=h'(s_2)=0$. Then we have
\[
\int_\R \frac{f(s_1u)}{s_1u}u^2dx=\int_\R \frac{f(s_2u)}{s_2u}u^2dx.
\]
Since by \ref{item:h5} $s\to\frac{f(s)}{s}$ is increasing, we have a
contradiction. Since $f(s)=o(s)$ for $s\to0$, we have $S(su)>0$ if
$s$ is small enough. On the other hand, \ref{item:h6} implies that 
\[
\partial_s\frac{F(s)}{s^2}=\frac{f(s)s-2F(s)}{s^3}>\frac{(\theta-2)}{s}\frac{F(s)}{s^2},
\] 
i.e. $F$ is superquadratic and therefore for $s$ large we must have
$S(su)<0$. Hence $h'$ vanishes exactly once at $s_u$, $h'(s)>0$ for
$s<s_u$ and $h'(s)<0$ for $s>s_u$. Moreover,
$h(s_u)=\max_{s\in(0,\infty)}h(s)$. Since $S(su)=h(s)$ and
$I(su)=sh'(s)$, this concludes the proof. 
\end{proof}

Define $\mathcal N_{\mathrm{rad}}:=\mathcal N\cap
H_{\mathrm{rad}}^1(\R^d)$. We have the following compactness result. 

\begin{lemma}\label{lem:nehari-seq}
  Let $d\geq 2$. Let $(u_n)\subset \mathcal N_{\mathrm{rad}}$ and assume that $S(u_n)$ is bounded. Then
  $(u_n)$ is bounded in $H^1(\R^d)$ and
there exists $u_\infty\in
  H_{\mathrm{rad}}^1(\R^d)\setminus\{0\}$ such that  (up to
  extraction of a subsequence) we have
\[
u_n\rightharpoonup u_\infty\text{ weakly in }H_{\mathrm{rad}}^1(\R).
\]
Moreover, there exists $s_\infty>0$ such that $s_\infty
u_\infty\in\mathcal N_{\mathrm{rad}}$
 and $S(s_\infty u_\infty)\leq
\liminf_{n\to\infty} S(u_n)$.
\end{lemma}

\begin{proof}
  Take a sequence $(u_n)\in \mathcal N_{\mathrm{rad}}$ and assume that $S(u_n)$ is
  bounded. Arguing by contradiction, we assume that
  $\nu_n:=\norm{u_n}_{H^1}\to\infty$. Define a sequence $(v_n)\subset
  H_{\mathrm{rad}}^1(\R^d)$ by 
\[
v_n:=\frac{u_n}{\nu_n}.
\]
Then $(v_n)$ is bounded in $H_{\mathrm{rad}}^1(\R^d)$ and  there exists
$v_\infty\in H_{\mathrm{rad}}^1(\R^d)$ such that $v_n\rightharpoonup v_\infty$ weakly
in $H_{\mathrm{rad}}^1(\R^d)$. We claim that $v_\infty\neq 0$. Arguing again by
contradiction, assume that $v_\infty=0$. By Lemma \ref{lem:scaling}, for any $s>0$ and $n$ large enough, we have
\[
S(u_n)=S(\nu_nv_n)\geq S(sv_n)=\frac{s^2}{2}-\int_\R F(sv_n)dx.
\]
Since $d\geq2$, the injection $H_{\mathrm{rad}}^1(\R^d)\hookrightarrow
L^q(\R^d)$ is compact for any $2<q<2^*$. Combined with
\ref{item:h2}, this implies weak continuity of $u\to
\int_{\R^d}F(u)dx$. Since we assumed $v_\infty=0$, for $n$ large we have
\[
S(u_n)\geq \frac{s^2}{4},
\]
which is a contradiction since $S(u_n)$ is bounded and $s$ can be
chosen a large as desired. Hence $v_\infty\neq 0$. Moreover, since $I(u_n)=0$,
we have
\begin{equation}
\label{eq:to-contradict}
0\leq \frac{S(u_n)}{\nu_n^2}=\frac12-\frac{1}{\nu_n^2}\int_{\R}F(\nu_n v_n)dx.
\end{equation}
From \ref{item:h6}, we have 
\(
F(s)\geq s^\theta,
\)
thus 
\[
\lim_{s\to\infty}\frac{F(s)}{s^2}=\infty.
\]
Recall that upon extraction of subsequences, $v_n\rightharpoonup v_\infty\neq
0$ weakly in $H_{\mathrm{rad}}^1(\R^d)$ and $v_n(x)\to v_\infty(x)$ a.e. By Fatou's Lemma,
this implies
\[
\frac{1}{\nu_n^2}\int_{\R}F(\nu_n v_n)dx=\int_{\R}\frac{F(\nu_n
  v_n)}{(\nu_n v_n)^2}v_n^2dx\to\infty\text{ as }n\to\infty.
\]
This leads to a contradiction in \eqref{eq:to-contradict}. Therefore,
$(\nu_n)=(\norm{u_n}_{H^1})$ has to remain bounded. As a consequence,
there exists $u_\infty\in H_{\mathrm{rad}}^1(\R^d)$ such that $u_n\rightharpoonup
u_\infty$. We can prove that $u_\infty\neq 0$ in the same way as we
did for $v_\infty$. 
Moreover, there exists $s_\infty$ such that
$s_\infty u_\infty\in\mathcal N_{\mathrm{rad}}$ and we have
\[
S(s_\infty u_\infty)\leq \liminf_{n\to\infty}S(s_\infty u_n)\leq \liminf_{n\to\infty}S(u_n).
\] 
This concludes the proof. 
\end{proof}

The proofs of Theorem \ref{thm:least-energy-nodal} and Theorem \ref{thm:least-energy-multinodal} follow from a similar line of arguments. We give the details for the proof of Theorem \ref{thm:least-energy-multinodal}, and we will only highlight the differences for the proof of Theorem \ref{thm:least-energy-nodal}.

\begin{proof}[Proof of Theorem \ref{thm:least-energy-multinodal}]
 Let $(u_n)$ be a minimizing sequence for $m_{\mathcal
   N_{N_{\textrm{nodes}}}}$, i.e. $(u_n)\subset \mathcal
   N_{N_{\textrm{nodes}}}$ and $S(u_n)\to m_{\mathcal
     N_{N_{\textrm{nodes}}}}$ as $n\to \infty$.

   Since $(u_n)\subset \mathcal
   N_{N_{\textrm{nodes}}}$, there exist $0=\rho_0<\rho_1^n<\dots<\rho_{N_{\textrm{nodes}}}^n<\rho_{N_{\textrm{nodes}}+1}=\infty$ such that $(\rho_1^n,\dots,\rho_{N_{\textrm{nodes}}}^n)\subset u_n^{-1}(0)$ and for all $k=0,\dots, N_{\textrm{nodes}}$
   \[
u_n^k:=  (u_n)_{|[\rho_k^n,\rho_{k+1}^n]}\in\mathcal N.
     \]
   By Lemma \ref{lem:nehari-seq} the sequence
 $(u_n^k)$ is bounded in $H_{\mathrm{rad}}^1(\R^d)$ and there
 exists $ u_\infty^k\in H_{\mathrm{rad}}^1(\R^d)\setminus\{0\}$ such that
 $u_n^k\rightharpoonup  u_\infty^k$ weakly in $H_{\mathrm{rad}}^1(\R^d)$
 and a.e. In particular, pointwise convergence implies
 $ u_\infty^j u_\infty^k=0$ a.e. whenever $j\neq k$. Let $s^k$ be
 such that $I(s^k u_\infty^k)=0$ and define 
\[
u_\infty=\sum_{k=0}^{N_{\textrm{nodes}}}s^k u_\infty^k.
\]
By construction $u_\infty\in\mathcal N_{N_{\textrm{nodes}}}$. Moreover, 
\[
S(u_\infty)=\sum_{k=0}^{N_{\textrm{nodes}}}S(s^k u_\infty^k)\leq
\liminf_{n\to\infty}\sum_{k=0}^{N_{\textrm{nodes}}}S(u_n^k)=\lim_{n\to\infty}S(u_n)=m_{\mathcal
N_{N_{\textrm{nodes}}}}.
\]
This proves  that $u_\infty$ is a minimizer for $m_{\mathcal
N_{N_{\textrm{nodes}}}}$. 

Let us show that in fact $s^k=1$ for $k=0,\dots, N_{\textrm{nodes}}$ and the sequence $(u_n)$ converges strongly in
$H_{\mathrm{rad}}^1(\R^d)$ toward $u_\infty$. By \ref{item:h2} and Strauss'
Lemma, the functional $u\mapsto \int_{\R^d}f(u)udx$ is weakly
continuous on $H_{\mathrm{rad}}^1(\R^d)$. This implies that for all  $k=0,\dots, N_{\textrm{nodes}}$, we have
\[
I( u_\infty^k)\leq \liminf_{n\to\infty} I(u_n^k)=0.
\]
Hence by Lemma \ref{lem:scaling}  we have $s^k\leq 1$. Moreover
\[
S(u_\infty) =m_{\mathcal N_{N_{\textrm{nodes}}}}= \lim_{n\to\infty}S(u_n)= \lim_{n\to\infty}\sum_{k=0}^{N_{\textrm{nodes}}}S(u_n^k)
\]
and weak continuity of the nonlinear part of $S$
implies
\begin{equation}
\label{eq:follows}
\norm{u_\infty}_{H^1}^2=\liminf_{n\to \infty}\sum_{k=0}^{N_{\textrm{nodes}}}\norm{u_n^k}_{H^1}^2.
\end{equation}
Moreover
\begin{multline*}
  \liminf_{n\to \infty}\sum_{k=0}^{N_{\textrm{nodes}}}\norm{u_n^k}_{H^1}^2
  \geq \sum_{k=0}^{N_{\textrm{nodes}}}\norm*{ u_\infty^k}_{H^1}^2
  =\sum_{k=0}^{N_{\textrm{nodes}}}\frac{1}{(s^k)^2}\norm{s^k u_\infty^k}_{H^1}^2
  \\
  \geq \frac{1}{\max_{k=0,\dots,N_{\textrm{nodes}}}(s^k)^2}
  \sum_{k=0}^{N_{\textrm{nodes}}}\norm{s^k u_\infty^k}_{H^1}^2
  =\frac{1}{\max_{k=0,\dots,N_{\textrm{nodes}}}(s^k)^2}\norm{u_\infty}_{H^1}^2\\
=\frac{1}{\max_{k=0,\dots,N_{\textrm{nodes}}}(s^k)^2} \liminf_{n\to \infty}\sum_{k=0}^{N_{\textrm{nodes}}}\norm{u_n^k}_{H^1}^2,
\end{multline*}
where the first inequality is from weak convergence and the last
equality follows from \eqref{eq:follows}.  Since we already know that $s^k\leq 1$, this
implies that $s^k=1$ for any $k=0,\dots,N_{\textrm{nodes}}$ and strong convergence of $(u_n)$ towards
$u_\infty$ in $H^1(\R^d)$. 

We now show that $u_\infty$ is a critical point of $S$. Recall that
$\mathcal N_{N_{\textrm{nodes}}}$ is not a manifold and we cannot use
a Lagrange multiplier argument for the minimizers of $m_{\mathcal
N_{N_{\textrm{nodes}}}}$. Instead, we shall use the quantitative
deformation lemma of Willem \cite[Lemma 2.3]{Wi96}, which we recall in
Appendix (see Lemma \ref{lem:willem}). 
Arguing by contradiction, we assume that $S'(u_\infty)\neq 0$. Then there exist $\delta,\mu>0$ such that
\[
\norm{v-u_\infty}_{H^1}\leq 3\delta\implies\norm{S'(v)}_{H^1}\geq \mu.
\]
Define $D=\left[\frac12,\frac32\right]^{N_{\textrm{nodes}}+1}$ and $g:D\to H^1(\mathbb R^d)$ by
\[
  g(s^0,\dots,s^{N_{\textrm{nodes}}})=\sum_{k=0}^{N_{\textrm{nodes}}}s^ku_\infty^k.
  \]
  Let $s^0,\dots,s^{N_{\textrm{nodes}}}\in D\setminus\{(1,\dots,1)\}$.
  Then from 
Lemma \ref{lem:scaling}
we infer that
\[
S(g(s^0,\dots,s^{N_{\textrm{nodes}}}))=\sum_{k=0}^{N_{\textrm{nodes}}}S(s^ku_\infty^k)<\sum_{k=0}^{N_{\textrm{nodes}}}S(u_\infty^k)=m_{\mathcal
N_{N_{\textrm{nodes}}}}.
\]
Consequently, $S(g(s^0,\dots,s^{N_{\textrm{nodes}}}))=m_{\mathcal
N_{N_{\textrm{nodes}}}}$ if and only if $(s^0,\dots,s^{N_{\textrm{nodes}}})=(1,\dots,1)$ and otherwise
$
S(g(s^0,\dots,s^{N_{\textrm{nodes}}}))<m_{\mathcal
N_{N_{\textrm{nodes}}}}.
$
Hence
\[
\beta:=\max_{\partial D} S\circ g<m_{\mathcal
N_{N_{\textrm{nodes}}}}.
\]
Let $\eps:=\min\left(\frac{m_{\mathcal
N_{N_{\textrm{nodes}}}}-\beta}{4},\frac{\mu \delta}{8}\right)$. The
deformation lemma \ref{lem:willem} gives us a deformation $\eta$
verifying 
\begin{itemize}
\item [(a)] $\eta(1,v)=v$ if $v\notin S^{-1}([m_{\mathcal
N_{N_{\textrm{nodes}}}}-2\eps,m_{\mathcal
N_{N_{\textrm{nodes}}}}+\eps])$,
\item [(b)] $S(\eta(1,v))\leq m_{\mathcal
N_{N_{\textrm{nodes}}}}-\eps$ for every $v\in H_{\mathrm{rad}}^1(\R^d)$
such that $\norm{v-u_\infty}_{H^1}\leq \delta$ and $S(v)\leq m_{\mathcal
N_{N_{\textrm{nodes}}}}+\eps$,
\item [(c)] $S(\eta(1,v))\leq S(v)$ for all $v\in
  H^1_{\mathrm{rad}}(\R^d)$. 
\end{itemize}
In particular, we have
\begin{equation}
\label{eq:sw1}
\max_{(s,t)\in D}S(\eta(1, g(s^0,\dots,s^{N_{\textrm{nodes}}})))<m_{\mathcal
N_{N_{\textrm{nodes}}}}.
\end{equation}
To obtain a contradiction we prove that $\eta(1,g(D))\cap \mathcal
N_{N_{\textrm{nodes}}}\neq \emptyset$. Define
\[
  h(s^0,\dots,s^{N_{\textrm{nodes}}}):=\eta(1, g(s^0,\dots,s^{N_{\textrm{nodes}}})).
  \]
Observe that the deformation $\eta$ can be chosen chosen small enough so that it does not affect the number of nodal components of $g(s^0,\dots,s^{N_{\textrm{nodes}}})$, and $h(s^0,\dots,s^{N_{\textrm{nodes}}})$ has the same number of nodal components. We denote these components by $h^k$, $k=0,\dots,N_{N_{\textrm{nodes}}}$. 
We also define
\[
  \begin{aligned}
    \psi_0 (s^0,\dots,s^{N_{\textrm{nodes}}})&:=\left(I(s^0 u_\infty^0),\dots,I(s^{N_{\textrm{nodes}}} u_\infty^{N_{\textrm{nodes}}})\right),\\
   \psi_1 (s^0,\dots,s^{N_{\textrm{nodes}}})&:=\left(I(h^0 (s^0,\dots,s^{N_{\textrm{nodes}}})),\dots,I(h^{N_{\textrm{nodes}}} (s^0,\dots,s^{N_{\textrm{nodes}}}))\right),
  \end{aligned}
\]
As $I(s u_\infty^k)>0$ (resp. $<0$) if $0<s<1$ (resp. $s>1$), the
degree of $\psi_0$ (see e.g. \cite{AmMa06} for the definition and
basic properties of the degree) is $\operatorname{Deg}(\psi_0,D,0)=1$.
From (a) and \eqref{eq:sw1}, we have $g=h$ on 
$\partial D$. Therefore, $\psi_0=\psi_1$ on $\partial D$, which 
implies
$\operatorname{Deg}(\psi_1,D,0)=\operatorname{Deg}(\psi_0,D,0)=1$. Therefore,
there exists $(s^0,\dots,s^{N_{\textrm{nodes}}})\in 
D$ such that $\psi_1 (s^0,\dots,s^{N_{\textrm{nodes}}})=0$. That means
$h (s^0,\dots,s^{N_{\textrm{nodes}}})\in\mathcal N_{N_{\textrm{nodes}}}$, a contradiction with
\eqref{eq:sw1} and the definition of $m_{\mathcal
N_{N_{\textrm{nodes}}}}$. Therefore $u_\infty$ is a critical point of
$S$. 

It remains to prove that $u_\infty$ has exactly $N_{\textrm{nodes}}$ nodes. Assume by the contrary that there exists $k\in\{0,\dots,N_{\textrm{nodes}}\}$ such that $u_\infty^k$ vanishes on $(\rho_k,\rho_{k+1})$. Then $u_\infty^k$ cannot be a minimizer of the action on the Nehari manifold restricted to $[\rho_k,\rho_{k+1}]$, i.e.
\[
S(u_\infty^k)>\min\{u\in H^1(H_{\textrm{rad}}^1(\Omega(\rho_k,\rho_{k+1}),\, u(\rho_k)=u(\rho_{k+1})=0\,I(u)=0\}.
\]
Indeed, by classical arguments, the minimizer can be shown to exist and to be strictly positive inside $\Omega(\rho_k,\rho_{k+1})$, see \cite{BaWe93}. Replacing $u_\infty^k$ by the minimizer in the construction of $u_\infty$, one would obtain a function verifying all required conditions, with a strictly lower energy. This is a contradiction, hence $u_\infty$ cannot vanish inside $(\rho_k,\rho_{k+1})$ and must have exactly $N_{\textrm{nodes}}$ nodes. 
\end{proof}

\begin{proof}[Proof of Theorem \ref{thm:least-energy-nodal}]

The first parts of the proof of Theorem \ref{thm:least-energy-nodal} follow from similar arguments as for Theorem \ref{thm:least-energy-multinodal}, with most of the modifications consisting simply in replacing the index/exponent $k$ by $\pm$. To prove that $u_\infty$ has exactly two nodal domains, the arguments slightly differ and we give some more details. 

Observe first that $u_\infty$ is also a minimizer for
\[
m_{alt} = \inf\left\{ \int_{\R^d}-F(u)+\frac12f(u)udx:I(u^+)\leq 0,\,I(u^-)\leq0,\, u^\pm\neq0\right\}.
\]
By \ref{item:h6}, $2F(s)-f(s)s<0$ for any $s$. As a consequence, the minimizers of $m_{alt}$ verify $I(u^+)=I(u^-)=0$. Indeed, arguing by contradiction and assuming e.g. $I(u^+)<0$, one could replace $u^+$ by $tu^+$ with $0<t<1$ such that $I(tu^+)=0$, which would give a minimizer of $m_{alt}$ for a lower value and provides the contradiction. Assume now that $u_\infty$  has more than one nodal region, i.e that there exists $s_1<s_2<s_3$ such that $u_\infty>0$ on $(0,s_1)$, $u_\infty<0$ on $(s_1,s_2)$ and $u_\infty>0$ on $(s_2,s_3)$. Let $u_1$ be such that $u_1=u_\infty$ on $(0,s_1)$ and $u_1=0$ elsewhere. Define similarly $u_2$ for $(s_1,s_2)$ and $u_3^\pm=u_\infty^\pm$,  on $(s_3,\infty)$. Since $I(u^+)=0$ we either have $I(u_1)\leq0$ or $I(u_3^+)\leq 0$. Without loss of generality, assume that $I(u_1)\leq0$. We may then construct $\tilde u_\infty$ such that $\tilde u_\infty=u_1$ on $(0,s_1)$,  $\tilde u_\infty=u_2$ on $(s_1,s_2)$ and $\tilde u_\infty=0$ on $(s_3,\infty)$. As it is not containing the $u_3^\pm$ parts, $\tilde u_\infty$ would be a minimizer for $m_{alt}$ for a lower value than $u_\infty$, which provides a contradiction and finishes the proof. 
\end{proof}

\section{The shooting method}
\label{sec:shooting}

We describe in this section the shooting method, its theoretical background and its practical implementation. Radial solutions of \eqref{eq:snls} can be obtained as solutions of the ordinary differential equation
\begin{equation}
  \label{eq:1}
  -u''(r)-\frac{d-1}{r}u'(r)+\omega u(r)-f(u(r))=0.
\end{equation}
They should satisfy the boundary conditions
\[
u'(0)=0, \quad \lim_{r\to\infty}u(r)=0.
  \]
  \begin{remark}
   Notice that \eqref{eq:1} does not fit directly into the standard Cauchy-Lipschitz theorem. However, the equation is equivalent to 
\[
(r^{d-1}u'(r))'=r^{d-1}(\omega u(r)-f(u(r))),
\]
which can reformulated (using the initial condition $u'(0)=0$) as
\[
\begin{cases}
u(r)=u(0)+\int_0^rv(s)ds,\\
v(r)=\frac{1}{r^{d-1}}\int_0^rs^{d-1}(\omega u(s)-f(u(s)))ds.
\end{cases}
\]
This system might be proven to have a unique solution by a standard Picard iteration scheme. 
  \end{remark}
  
  It was established in \cite{CoGaYa11} that, under convexity assumptions on $f$ (which hold for example in dimension $d=2$ when $f$ is of power-type), that for any $k\in\mathbb N$, the equation \eqref{eq:1} admits exactly one solution having exactly $k$ nodes. More precisely, it was proved in \cite{CoGaYa11} that there exists an increasing sequence $(\alpha_k)\subset (0,\infty)$ such that for any $k\in\mathbb N$ and for any $\alpha\in (\alpha_{k-1},\alpha_k)$ (with the understanding that $\alpha_{-1}=0$), the solution of the Cauchy Problem
  \begin{equation}
    \label{eq:shoot_to_thrill}
  -u''(r)-\frac{d-1}{r}u'(r)+\omega u(r)-f(u(r))=0, \quad u(0)=\alpha,\quad u'(0)=0,
\end{equation}
denoted by $u(\cdot;\alpha)$, has exactly $k$ nodes on $[0,\infty)$.

Moreover, when (and only when) $\alpha=\alpha_k$, the solution  $u(\cdot;\alpha_k)$ of \eqref{eq:shoot_to_thrill} verifies
\[
  \lim_{r\to\infty} u(r;\alpha_k)=0.
\]
In other words, finding the $k$-th radial state amounts to finding the corresponding $\alpha_k$. We expect that $k\sim \alpha_k^2$ (see Figure \ref{fig:nodsdata}). Observe that when $\alpha\in(\alpha_k,\alpha_{k+1})$, then the solution of the differential equation such that  $u(0)=\alpha$ has exactly $k$ nodes (but does not converge to $0$ at infinity).

The so-called shooting method then consists in a simple application of the bisection principle to the search of $\alpha_k$. The idea is the following. Start with an interval $[\alpha_*,\alpha^*]$ such that 
$u(\cdot;\alpha_*)$ and  $u(\cdot;\alpha^*)$ have respectively $k-1$ and $k$ nodes on $(0,\infty)$. Define the middle of $[\alpha_*,\alpha^*]$ by $c_*=(\alpha_*+\alpha^*)/2$. If the solution $u(\cdot;c_*)$ with initial data $c_*$  has $k-1$ nodes on $(0,\infty)$, then reproduce the procedure on $[c_*,\alpha^*]$, otherwise iterate the procedure on $[\alpha_*,c_*]$. The interval size is divided by two at each step and its bounds converge towards $\alpha_k$.

To compute the solution of \eqref{eq:1}, we rewrite it as a first order system
  \begin{equation}
    \label{eq:2}
U'(r)=
AU+G(r,U(r)),\quad
U(0)=
\begin{pmatrix}
  \alpha\\0
\end{pmatrix}
\end{equation}
where
\[
U(r)=\binom{u(r)}{u'(r)},\quad
A=\begin{pmatrix}
  0&1\\
  \omega&0
\end{pmatrix},\quad
G(r,U(r))=
\begin{pmatrix}
  0\\
  \frac{d-1}{r}u'(r)+f(u(r))
\end{pmatrix}. 
  \]
  The solution of the initial value problem \eqref{eq:2} is then computed using the classical Runge-Kutta 4th order method. The only difficulty concerns the value to affect to $G$ at $r=0$. The singularity can in fact be raised when $u$ is sufficiently regular and the initial condition contains $u'(0)=0$. Indeed, assuming that $u'\in\mathcal C^2$ and writing the Taylor expansion at $0$, we have
  \[
u'(r)=u'(0)+u''(0)r+u'''(\theta)\frac{r^2}{2}=u''(0)r+u'''(\theta)\frac{r^2}{2},\quad \theta\in(0,r).
\]
Assuming that $u$ verifies \eqref{eq:1}, we have
\[
-u''(r)-(d-1)u''(0)+u'''(\theta)\frac{(d-1)r}{2}+\omega u(r)-f(u(r))=0.
  \]
  Letting $r$ tend to $0$, we obtain
  \[
-du''(0)+\omega u(0)-f(u(0))=0.
\]
Therefore, we choose to set
\[
G\left(0,\binom{u(0)}{0}\right)=\binom{0}{(\omega u(0)-f(u(0)))/d}.
  \]




\begin{algorithm}[!h]
\caption{The Shooting Method Algorithm. \label{alg:shooting}}
\begin{algorithmic} 
  \REQUIRE
$G$, $\eps$, $\texttt{nb\_expected\_nodes}$, $b>0$
\STATE{$a\gets 0$}
\WHILE{$|a-b|<\eps$}
\STATE{$c\gets (a+b)/2$}
  \STATE{\texttt{Solve }$Y'=G(t,Y)$, $Y(0)=c$ with RK4}
  \STATE{$\texttt{nb\_nodes}\gets$ number of nodes of $Y$}
  \IF{$\texttt{nb\_nodes}>\texttt{nb\_expected\_nodes}$}
  \STATE{$b\gets c$}
  \ELSE
  \STATE{$a\gets c$}
  \ENDIF
  \ENDWHILE
 \end{algorithmic}
\end{algorithm}

The algorithm is given in Algorithm \ref{alg:shooting}. 
 Note that in Algorithm \ref{alg:shooting}, it is understood that $b$ has been chosen large enough so that the initial data of the excited state with the required number of nodes lies in $[0,b]$. In pratice, such a $b$ can be obtained by inspection, taking larger and larger values until the solution of \eqref{eq:shoot_to_thrill} with initial data $\alpha=b$ has a sufficiently large number of nodes. 

Examples of solutions computed with the Shooting method are presented in Figure \ref{ref:boundstatesd3}.

\begin{figure}
 \centering
     \begin{subfigure}[b]{0.45\textwidth}
         \centering
         \includegraphics[width=\textwidth]{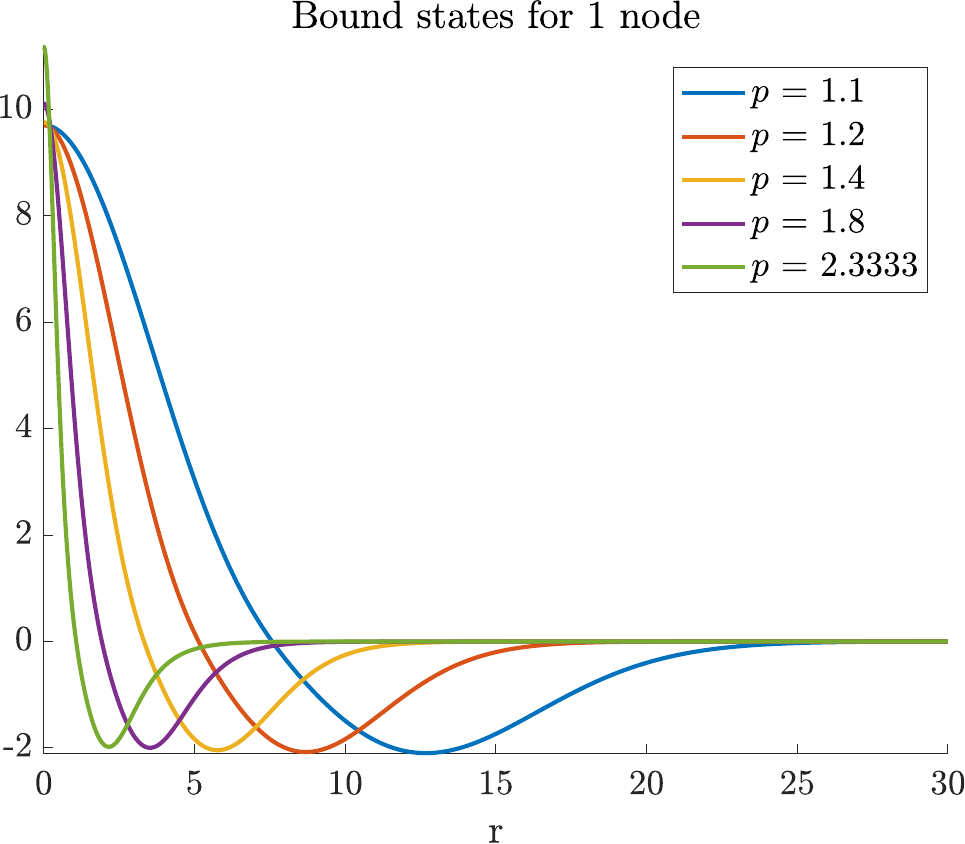}
         \caption{$N_{\textrm{nodes}} = 1$}
     \end{subfigure}
     \hspace{2em}
     \begin{subfigure}[b]{0.45\textwidth}
         \centering
         \includegraphics[width=\textwidth]{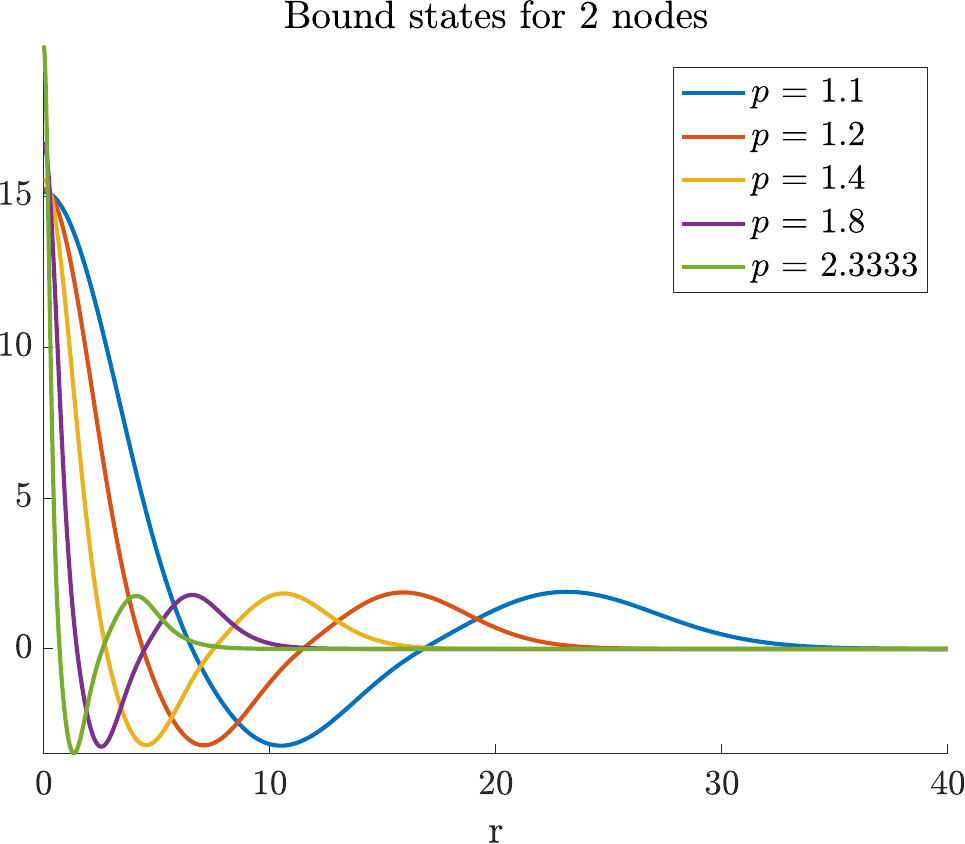}
         \caption{$N_{\textrm{nodes}} = 2$}
     \end{subfigure}
     
     \begin{subfigure}[b]{0.45\textwidth}
         \centering
         \includegraphics[width=\textwidth]{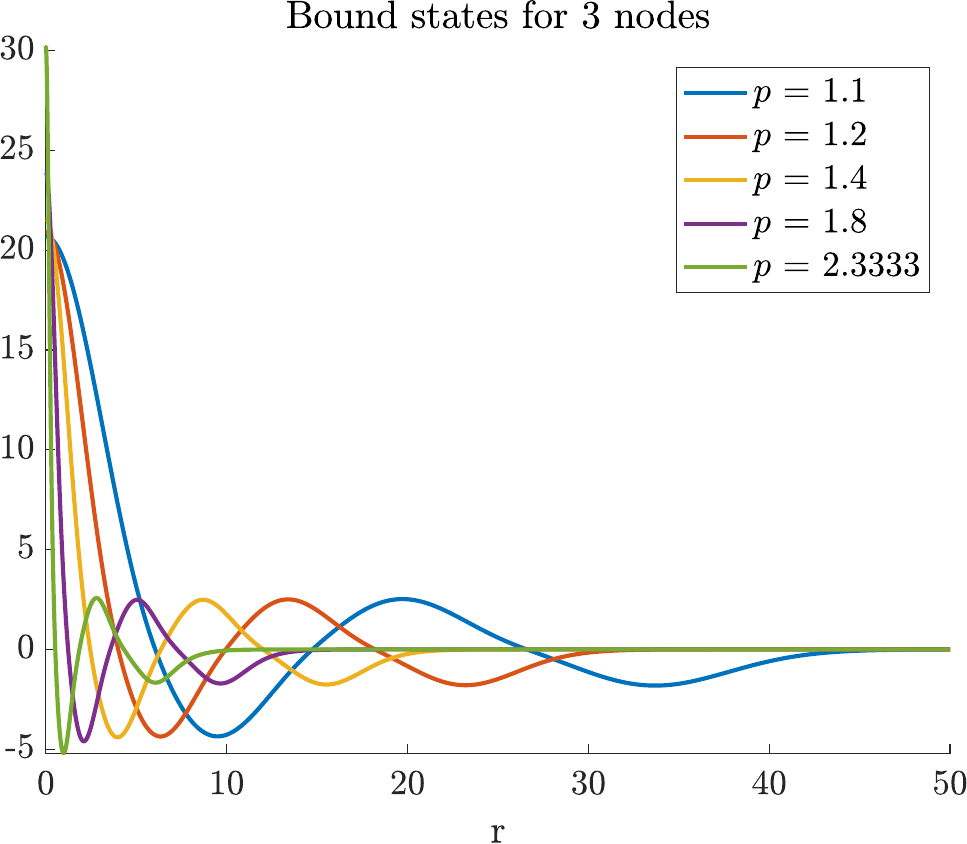}
         \caption{$N_{\textrm{nodes}} = 3$}
     \end{subfigure}
     \hspace{2em}
     \begin{subfigure}[b]{0.45\textwidth}
         \centering
         \includegraphics[width=\textwidth]{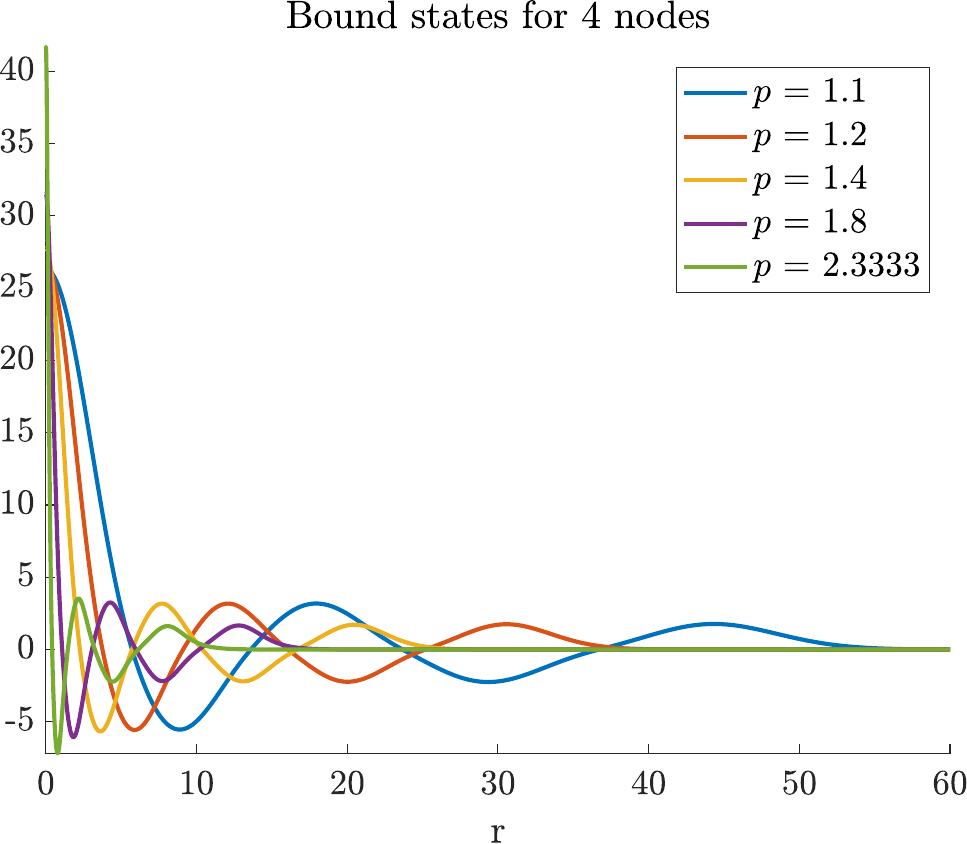}
         \caption{$N_{\textrm{nodes}} = 4$}
     \end{subfigure}

        \caption{Bound states in dimension $d=3$}\label{ref:boundstatesd3}
\end{figure}

\section{The Nehari method}
\label{sec:nehari}

In this section, we assume that $f(u) = |u|^{p-1}u$ with $p\in (1,1+4/(d-2)_+)$ and we consider on $\R^d$ the problem 
\begin{equation*}\label{eq:snls-2}
     -\Delta u+u-|u|^{p-1}u=0.
\end{equation*}
We want to compute numerically the sequence from Theorem \ref{thm:least-energy-multinodal}. 

For practical implementation, we must restrict ourselves to a bounded interval. That is, instead of considering the problem on the whole line $\mathbb R$, we restrict ourselves to the interval $[0,R]$ for a given $R>0$ sufficiently large. 
We consider the space
\begin{equation*}
\mathcal{H}_{\textrm{rad},R}^1 : = 
 \left\{ u: [0,R]\mapsto \mathbb{R}: \int_0^{R}( |u'(r)|^2 + |u(r)|^2)r^{d-1} dr <+\infty,\;  u(R) = 0\right\},
\end{equation*}
on which we define the functionals $S$ and $I$ by the same formula, but restricted to the interval $[0,R]$.
Similarly, we define the nodal Nehari space
\begin{align*}
&\mathcal{N}_{N_{\textrm{nodes}},R} : = 
\\ &\hspace{0.5em}\left\{u \in \mathcal{H}_{\textrm{rad},R}^1: u^{-1}(0) = \{\rho_1,\rho_2,\ldots, \rho_{N_{\textrm{nodes}}} \},\; I(u_{|[\rho_k,\rho_{k+1}]}) = 0,\; 0\leq k \leq N_{\textrm{nodes}}  \right\},
\end{align*}
where $0 = \rho_0 < \rho_1<\ldots<\rho_{N_{\textrm{nodes}}+1} = R$ are depending on $u$ and we assume that $u$ changes sign at each node. Let
\begin{equation*}
\mathcal{H}_{N_{\textrm{nodes}},R} : = \left\{u \in \mathcal{H}_{\textrm{rad},R}^1\middle|\; u^{-1}(0) = \{\rho_1,\rho_2,\ldots, \rho_{N_{\textrm{nodes}}} \}\right\}.
\end{equation*}
Then, we define the projection $\Pi_{\mathcal{N}_{N_{\textrm{nodes}}},R}: \mathcal{H}_{N_{\textrm{nodes}},R} \mapsto \mathcal{N}_{N_{\textrm{nodes}},R}$ by
\begin{equation*}
\Pi_{\mathcal{N}_{N_{\textrm{nodes}}}} u : = \sum_{k = 0}^{N_{\textrm{nodes}}}u_{|[\rho_k,\rho_{k+1}]} \left( \frac{\|\nabla u_{|[\rho_k,\rho_{k+1}]}\|^2_{L^2} + \|u_{|[\rho_k,\rho_{k+1}]}\|^2_{L^2}}{\|u_{|[\rho_k,\rho_{k+1}]}\|_{L^{p+1}}^{p+1}} \right)^{1/(p-1)}.
\end{equation*}

To construct a minimizing sequence for the problem
\begin{equation}\label{eq:discMinR}
\min_{u\in \mathcal{N}_{N_{\textrm{nodes}},R}}S(u),
\end{equation}
a natural method is to use the so-called projected gradient descent given by
\begin{equation*}
\left\{\begin{array}{ll}
u^{(0)} \in  \mathcal{N}_{N_{\textrm{nodes}},R},
\\ u^{(n+1)} = \Pi_{\mathcal{N}_{N_{\textrm{nodes}},R}} \left(u^{(n)} + \tau S'(u^{(n)})\right),\quad \forall n\geq 0,
\end{array}\right.
\end{equation*}
where $\tau \in\mathbb{R}^+$ is the time-step, which also writes as
\begin{equation*}
\left\{\begin{array}{ll}
u^{(0)} \in  \mathcal{N}_{N_{\textrm{nodes}},R},
\\ u^{(n+1)} = \Pi_{\mathcal{N}_{N_{\textrm{nodes}},R}} \left(u^{(n)} - \tau(\Delta_{\textrm{rad},R}u^{(n)} -u^{(n)}  + |u^{(n)}|^{p-1}u^{(n)})\right),\quad \forall n\geq 0.
\end{array}\right.
\end{equation*}
\begin{remark}
This algorithm is designed in the following way: first a gradient step is performed on the whole function, then each nodal component of the function is projected on the corresponding Nehari manifold. In particular, the motion of the nodes is determined during the gradient step. Furthermore, in practice, we observe that the total number of nodes is conserved. In the context of parabolic flows, this is an expected property which, for instance, has been proved in \cite{Ia13}.
\end{remark}
We now proceed with the spatial discretization. By setting $\pi_{N}([0,R]) : = \{r_k := (k-1)h,\;1\leq k\leq N+1\}$ with $h = R/N$, we can consider a discretization of $\Delta_{\textrm{rad},R}$ by finite differences acting on $\mathbb{R}^{N}$. That is, for any $u\in \mathcal{H}_{\textrm{rad},R}^1$ we use the second order approximations, for any $2\leq k\leq N-1$,
\begin{equation}
u''(r_k) \approx \frac{u(r_{k+1}) - 2 u(r_{k}) + u(r_{k-1})}{h^2}\quad\mbox{and}\quad u'(r_k) \approx \frac{u(r_{k+1}) - u(r_{k-1})}{2h},\label{eq:finitediff}
\end{equation}
to deduce the following approximation
\begin{equation*}
\Delta_{\textrm{rad},R} u(r_k) \approx \frac{u(r_{k+1}) - 2 u(r_{k}) + u(r_{k-1})}{h^2} + \frac{d-1}{r_k}\frac{u(r_{k+1}) - u(r_{k-1})}{2h}.
\end{equation*}
Furthermore, the boundary conditions yield
\begin{align*}
\Delta_{\textrm{rad},R} u (r_1) &\approx \frac{2(u(r_2)-u(r_1))}{h^2}
\\ \mbox{and}\quad \Delta_{\textrm{rad},R} u (r_N) &\approx\frac{ - 2 u(r_{N}) + u(r_{N-1})}{h^2} - \frac{d-1}{r_N}\frac{u(r_{N-1})}{2h}.
\end{align*}
In the end, we obtain the matrix
\begin{equation*}
[\boldsymbol\Delta_{\textrm{rad},R}]_{i,j} : = \left\{
\begin{array}{ll} 
 2/h^2,\quad\mbox{for }(i,j) = (1,2),
\\ -2/h^2,\quad\mbox{for } 1\leq i\leq  N\mbox{ and }j = i,
\\ 1/h^2 - (d-1)/2hr_{i},\quad\mbox{for } 2\leq i\leq  N\mbox{ and }j = i - 1,
\\ 1/h^2 + (d-1)/2hr_{i},\quad\mbox{for } 2\leq i\leq  N-1\mbox{ and }j = i + 1,
\\ 0,\quad\mbox{else}.
\end{array}\right.
\end{equation*}
By denoting $\boldsymbol{u} = (u(r_j))_{1\leq j\leq N}\in\mathbb{R}^N$ as the discretization of $u$ on $\pi_N([0,R])$, we deduce that
\begin{equation*}
\Delta_{\textrm{rad},R} u(r_i) \approx ([\boldsymbol\Delta_{\textrm{rad},R}]\boldsymbol{u})_i,\quad\forall i\in\{1,\ldots,N\}.
\end{equation*}
We also need to discretize the positions of the nodes $\{\rho_1,\ldots,\rho_{N_\textrm{nodes}}\}$ and the functionals involved. For any $u\in \mathcal{H}_{\textrm{rad},R}^1$, we assume that each node of $u$ is located in an interval $(r_j,r_{j+1})$, for a certain $1\leq j\leq N$, where $u(r_{j+1})u(r_j)<0$ (note that here the term \emph{node} refers to a point where the function changes sign). By a linear approximation of $u$ on each interval $[r_j,r_{j+1}]$, with $1\leq j\leq N$, an approximation of a node $\rho$ belonging in $[r_j,r_{j+1}]$ will be given by
\begin{equation}\label{eq:approxnod}
\rho \approx \varrho = \frac{r_j u(r_{j+1}) - r_{j+1} u(r_{j})}{u(r_{j+1})-u(r_{j})}.
\end{equation}
We now turn to the discretization of the functionals and choose to rely on a trapezoidal rule. This yields, for any $v\in\mathcal{C}([0,R])$,
\begin{align*}
\int_{a}^{b} v(r) r^{d-1}dr &\approx \frac{v(b)b^{d-1} + v(a)a^{d-1}}2 (b-a).
\end{align*}
We denote $(\rho_k)_{0\leq k\leq N_{\textrm{nodes}}+1}$ the nodes of $u$ with $\rho_0 = 0$ and $\rho_{N_{\textrm{nodes}}+1} = R$. For any $u\in\mathcal{H}_{\textrm{rad},R}^1$ and $0\leq k\leq N_{\textrm{nodes}}$, we deduce the approximation (using the trapezoidal rule)
\begin{align}
\|u_{|[\rho_k,\rho_{k+1}]}\|_{L^p}^p &= \int_{\rho_k}^{\rho_{k+1}} |u(r)|^p r^{d-1}dr \nonumber
\\ &\approx  h \sum_{j = m(\varrho_k)}^{\ell(\varrho_{k+1})} |u(r_j)|^pr_j^{d-1} + \frac{(r_{m(\varrho_k)} - \varrho_k )- h}2 |u(r_{m(\varrho_k)})|^pr_{m(\varrho_k)}^{d-1}\nonumber
\\ &\hspace{2em} + \frac{(\varrho_{k+1} - r_{\ell(\varrho_{k+1})})- h}2 |u(r_{\ell(\varrho_{k+1})})|^pr_{\ell(\varrho_{k+1})}^{d-1} = : \mathfrak{L}^p_k(\boldsymbol{u}),\label{eq:approxLp}
\end{align}
where $m(\varrho) = \min\{j\in\{1,\ldots,N\}:\; r_j \geq \varrho\}$, $\ell(\varrho) =  \max\{j\in\{1,\ldots,N\}:\; r_j\leq \varrho\}$ and $\varrho_k$ is the approximation of $\rho_k$ obtained by \eqref{eq:approxnod}. Furthermore, we have, by using \eqref{eq:finitediff}, for any $u\in\mathcal{H}_{\textrm{rad},R}^1$ and $1\leq k\leq N_{\textrm{nodes}}-1$,
\begin{align}
&\|\nabla u_{|[\rho_k,\rho_{k+1}]}\|_{L^2}^2 \approx h \sum_{j = m(\varrho_k)}^{\ell(\varrho_{k+1})} \left|\frac{u(r_{j+1}) - u(r_{j-1})}{2h}\right|^2r_j^{d-1}\nonumber
\\ &\hspace{1em} + \frac{(r_{m(\varrho_k)} - \varrho_k )- h}2 \left|\frac{u(r_{m(\varrho_k)+1})-u(r_{m(\varrho_k)-1})}{2h}\right|^2r_{m(\varrho_k)}^{d-1}\nonumber
\\ &\hspace{1em} + \frac{r_{m(\varrho_k)} - \varrho_k }{2} \left|\frac{u(r_{m(\varrho_{k})})-u(r_{m(\varrho_{k})-1})}{h}\right|^2\varrho_k^{d-1}\nonumber
\\ &\hspace{1em} + \frac{(\varrho_{k+1} - r_{\ell(\varrho_{k+1})})- h}2 \left|\frac{u(r_{\ell(\varrho_{k+1})+1})-u(r_{\ell(\varrho_{k+1})-1})}{2h}\right|^2r_{\ell(\varrho_{k+1})}^{d-1}\nonumber
\\ &\hspace{1em} + \frac{\varrho_{k+1} - r_{\ell(\varrho_{k+1})}}{2} \left|\frac{u(r_{\ell(\varrho_{k+1})+1})-u(r_{\ell(\varrho_{k+1})})}{h}\right|^2\varrho_{k+1}^{d-1}\nonumber
\\ &\hspace{6em} = : \mathfrak{N}_k(\boldsymbol{u}),\label{eq:approxG}
\end{align}
where we used the following finite differences approximation
\begin{equation*}
u'(\rho) \approx \frac{u(r_j)- u(r_{j-1})}{r_j - r_{j-1}},
\end{equation*}
with $j\in\{1,\ldots,N\}$ such that $\rho$ is a nod belonging in $(r_{j-1},r_j)$. We notice that, in the case $k = N_{\textrm{nodes}}$, the previous expression is replaced with
\begin{align}
&\|\nabla u_{|[\rho_{N_{\textrm{nodes}}},R]}\|_{L^2}^2 \approx h \sum_{j = m(\varrho_{N_{\textrm{nodes}}})}^{N-1} \left|\frac{u(r_{j+1}) - u(r_{j-1})}{2h}\right|^2r_j^{d-1}\nonumber
\\ &\hspace{1em} + \frac{(r_{m(\varrho_{N_{\textrm{nodes}}})} - \varrho_{N_{\textrm{nodes}}} )- h}2 \left|\frac{u(r_{m(\varrho_{N_{\textrm{nodes}}})+1})-u(r_{m(\varrho_{N_{\textrm{nodes}}})-1})}{2h}\right|^2r_{m(\varrho_{N_{\textrm{nodes}}})}^{d-1}\nonumber
\\ &\hspace{1em} + \frac{r_{m(\varrho_{N_{\textrm{nodes}}})} - \varrho_{N_{\textrm{nodes}}} }{2} \left|\frac{u(r_{m(\varrho_{N_{\textrm{nodes}}})})-u(r_{m(\varrho_{N_{\textrm{nodes}}})-1})}{h}\right|^2\varrho_{N_{\textrm{nodes}}}^{d-1}\nonumber
\\ &\hspace{1em} + \frac h2\left|\frac{u(r_{N-1})}{2h}\right|^2r^{d-1}_N\nonumber
\\ &\hspace{6em}  = : \mathfrak{N}_{N_{\textrm{nodes}}}(\boldsymbol{u}).\label{eq:approxGbis1}
\end{align}
For the case $k = 0$, we use instead
\begin{align}
&\|\nabla u_{|[0,\rho_{1}]}\|_{L^2}^2 \approx h \sum_{j = 2}^{\ell(\varrho_{1})} \left|\frac{u(r_{j+1}) - u(r_{j-1})}{2h}\right|^2r_j^{d-1}\nonumber
\\ &\hspace{1em} + \frac{(\varrho_{1} - r_{\ell(\varrho_{1})})- h}2 \left|\frac{u(r_{\ell(\varrho_{1})+1})-u(r_{\ell(\varrho_{1})-1})}{2h}\right|^2r_{\ell(\varrho_{1})}^{d-1}\nonumber
\\ &\hspace{1em} + \frac{\varrho_{1} - r_{\ell(\varrho_{1})}}{2} \left|\frac{u(r_{\ell(\varrho_{1})+1})-u(r_{\ell(\varrho_{1})})}{h}\right|^2\varrho_{1}^{d-1}\nonumber
\\ &\hspace{6em}  = : \mathfrak{N}_0(\boldsymbol{u}).\label{eq:approxGbis2}
\end{align}

Thanks to \eqref{eq:approxLp} and \eqref{eq:approxG}-\eqref{eq:approxGbis1}-\eqref{eq:approxGbis2}, we can deduce an approximation of the functionals. For the sake of simplicity, we do not take into account the exceptional case where a zero of a discretized function falls precisely on a point of the grid. We define the set of vectors $\mathbb{R}^N$ with $N_{\textrm{nodes}}$ nodes as
\begin{equation*}
    \mathbb{R}_{N_{\textrm{nodes}}}^{N} = \left\{\boldsymbol{u}\in(\mathbb{R}\setminus\{0\})^N:\#\left\{j=1,\ldots,N-1: {\boldsymbol{u}}_{j}{\boldsymbol{u}}_{j+1}< 0\right\} = N_{\textrm{nodes}} \right\}.
\end{equation*}
Furthermore, for any $\boldsymbol{u}\in \mathbb{R}_{N_{\textrm{nodes}}}^{N}$, we define
\begin{align*}
    \boldsymbol{m}(\boldsymbol{u}) &= \left\{j=1,...,N-1: \boldsymbol{u}_{j}\boldsymbol{u}_{j+1}< 0\right\}\in\mathbb{N}^{N_{\textrm{nodes}}},\\
        \boldsymbol{\ell}(\boldsymbol{u}) &= \left\{j=2,...,N: \boldsymbol{u}_{j-1}\boldsymbol{u}_{j}< 0\right\}\in\mathbb{N}^{N_{\textrm{nodes}}}.
\end{align*}
We deduce that the restricted Nehari functionals writes, for any $\boldsymbol{u}\in\mathbb{R}_{N_{\textrm{nodes}}}^{N}$ and any $0\leq k\leq N_{\textrm{nodes}}$, as
\begin{equation*} 
    \mathfrak{I}_k(\boldsymbol{u}) = \mathfrak{N}_k(\boldsymbol{u}) + \mathfrak{L}^2_k(\boldsymbol{u}) - \mathfrak{L}^{p+1}_k(\boldsymbol{u}),
\end{equation*}
as well as the total action
\begin{equation*}
    \mathfrak{S}(\boldsymbol{u}) = \sum_{k = 0}^{N_{\textrm{nodes}}} \frac12\mathfrak{N}_k(\boldsymbol{u}) + \frac12\mathfrak{L}^2_k(\boldsymbol{u}) - \frac1{p+1}\mathfrak{L}^{p+1}_k(\boldsymbol{u}).
\end{equation*}
By setting the discrete nodal Nehari manifold as
\begin{equation*}
    \mathcal{N}_{N_{\textrm{nodes}},R}^N = \left\{\boldsymbol{u}\in \mathbb{R}_{N_{\textrm{nodes}}}^{N}: \mathfrak{I}_k(\boldsymbol{u}) = 0,\; k = 0,\ldots,N_{\textrm{nodes}} \right\},
\end{equation*}
the discretization of the minimization problem \eqref{eq:discMinR} writes as
\begin{equation*}
    \min_{\boldsymbol{u}\in\mathcal{N}_{N_{\textrm{nodes}},R}^N }\mathfrak{S}(\boldsymbol{u}).
\end{equation*}
The solution of the above problem is obtained by a projected gradient method. By denoting $\mathfrak{E}_k(\boldsymbol{u}) = \frac12 \mathfrak{N}_k(\boldsymbol{u}) + \frac12\mathfrak{L}^{2}_k(\boldsymbol{u})$, we obtain the projection $\mathfrak{P}_{N_{\textrm{nodes}}}$ on the space $\mathcal{N}^N_{N_{\textrm{nodes}},R}$ of any vector $\boldsymbol{u}\in\mathbb{R}_{N_{\textrm{nodes}}}^{N}$:
\begin{equation*}
\mathfrak{P}_{N_{\textrm{nodes}}}\boldsymbol{u} = \sum_{k = 0}^{N_{\textrm{nodes}}} \boldsymbol{u}_{|\{\boldsymbol{m}(\boldsymbol{u})_k, \boldsymbol{\ell}(\boldsymbol{u})_k\}} \left(\frac{\mathfrak{E}_k(\boldsymbol{u})}{\mathfrak{L}^{p+1}_k(\boldsymbol{u})} \right)^{\frac{1}{p-1}},
\end{equation*}
where, for $1\leq j\leq N$,
\begin{equation*}
\left(\boldsymbol{u}_{|\{\boldsymbol{m}(\boldsymbol{u})_k, \boldsymbol{\ell}(\boldsymbol{u})_k\}}\right)_{j}  = \left\{\begin{array}{ll}
\boldsymbol{u}_j,\quad\mbox{if}\;\;\boldsymbol{m}(\boldsymbol{u})_k\leq j\leq \boldsymbol{\ell}(\boldsymbol{u})_k,
\\ 0,\quad\mbox{ else}.
\end{array}\right.
\end{equation*}
We can now gives a completely discretized version of the projected gradient descent method which is described in Algorithm \ref{alg:PGM} (where $[\boldsymbol{v}]_{i,j} = \boldsymbol{v}_j$ if $i = j$ and $0$ if $i\neq j$).

\begin{algorithm}[!h]
\caption{The projected gradient descent method \label{alg:PGM}.}
\begin{algorithmic} 
  \REQUIRE{$R,>0,\boldsymbol{u}^{(0)}\in\mathbb{R}^N,\tau>0, \varepsilon>0$}
\STATE{$\textrm{Crit} \gets 2\varepsilon  $}
\STATE{$j\gets 0$}
\WHILE{$\textrm{Crit} > \eps$}
\STATE{$\boldsymbol{v} \gets (\textrm{Id} - \tau ([\boldsymbol\Delta_{\textrm{rad},R}] + [|\boldsymbol{u}^{(j)}|^{p-1}]))\boldsymbol{u}^{(j)}$}
\STATE{$\boldsymbol{u}^{(j+1)}\gets \mathfrak{P}_{N_{\textrm{nodes}}}\boldsymbol{v}$}
\STATE{$\textrm{Crit}\gets \max_{1\leq \ell\leq N}|\boldsymbol{u}^{(j+1)}_{\ell} - \boldsymbol{u}^{(j)}_{\ell}|$}
\STATE{$j\gets j+1$}
\ENDWHILE
 \end{algorithmic}
\end{algorithm}

Note that we chose stagnation of the absolute error between two iterations as a stopping criterion. This is giving good results in the experiments, but other criteria such as the evolution of the action could have been considered.

\section{Some properties of the Nehari and shooting methods}
\label{sec:comparison}

In this section, we discuss some properties of the methods that we have introduced. Our goal is to point out some of their strengths and weaknesses.

In the case of the shooting method, observe that there is an inherent numerical difficulty associated with its practical implementation. Indeed, given $k\in\mathbb N$, the value of $\alpha_k$ can be determined only up to machine precision, i.e. $10^{-16}$ in practice. This is limiting the  size of the domain in $x$ on which  $u(\cdot;\alpha_k)$ can be computed accurately, even assuming no error on the numerical resolution of the Cauchy problem \eqref{eq:shoot_to_thrill}. Indeed, let $\eps>0$ and define $w_\eps=u(\cdot;\alpha_k)-u(\cdot;\alpha_k+\eps)$. Then $w_\eps$ verifies
\[
-w_\eps''+w_\eps-f'(u(\cdot;\alpha_k))w_\eps=O(w_\eps^2).
  \]
  As $\lim_{r\to\infty}f'(u(r;\alpha_k))=0$, the linear part of the equation is given by $-w_\eps''+w_\eps$. Whenever $w_\eps$ become small enough so that $O(w_\eps^2)$ becomes negligible, the dynamics of the equation of $w_\eps$ becomes driven by the linear part, for which $0$ is an exponentially unstable solution. As a consequence, we may have $w_\eps(x)\sim \eps e^{ x}$, which leads to $w_\eps\sim 1$ after $x\sim -\ln(\eps)$ (after which nonlinear effects cannot be neglected any more). For $\eps=  10^{-16}$, the best we can hope (assuming that the numerical method used to solve the ordinary differential equation is perfectly accurate) is therefore to solve our equation on an interval of length $-\ln(\eps)\sim 36$. This is illustrated in Figure \ref{fig:large-space-error}, on which we calculate the ground state with the shooting and Nehari methods in the case of the dimension $d = 2$ and for $p = 3$ when $R = 100$. We observe on the log-graph that at a distance from the origine around $19$, the calculated solution starts to increase and goes far away from the expected solution (which is exponentially decreasing toward $0$ at infinity). This issue is not observe in the case of the Nehari method due to the fact that we implement a Dirichlet boundary condition directly in the operator. This ensures that the numerical solution decreases properly to zero at the end of the domain.

\begin{figure}[htpb!]
 \begin{subfigure}[b]{0.45\textwidth}
         \centering
    \includegraphics[width=1\textwidth]{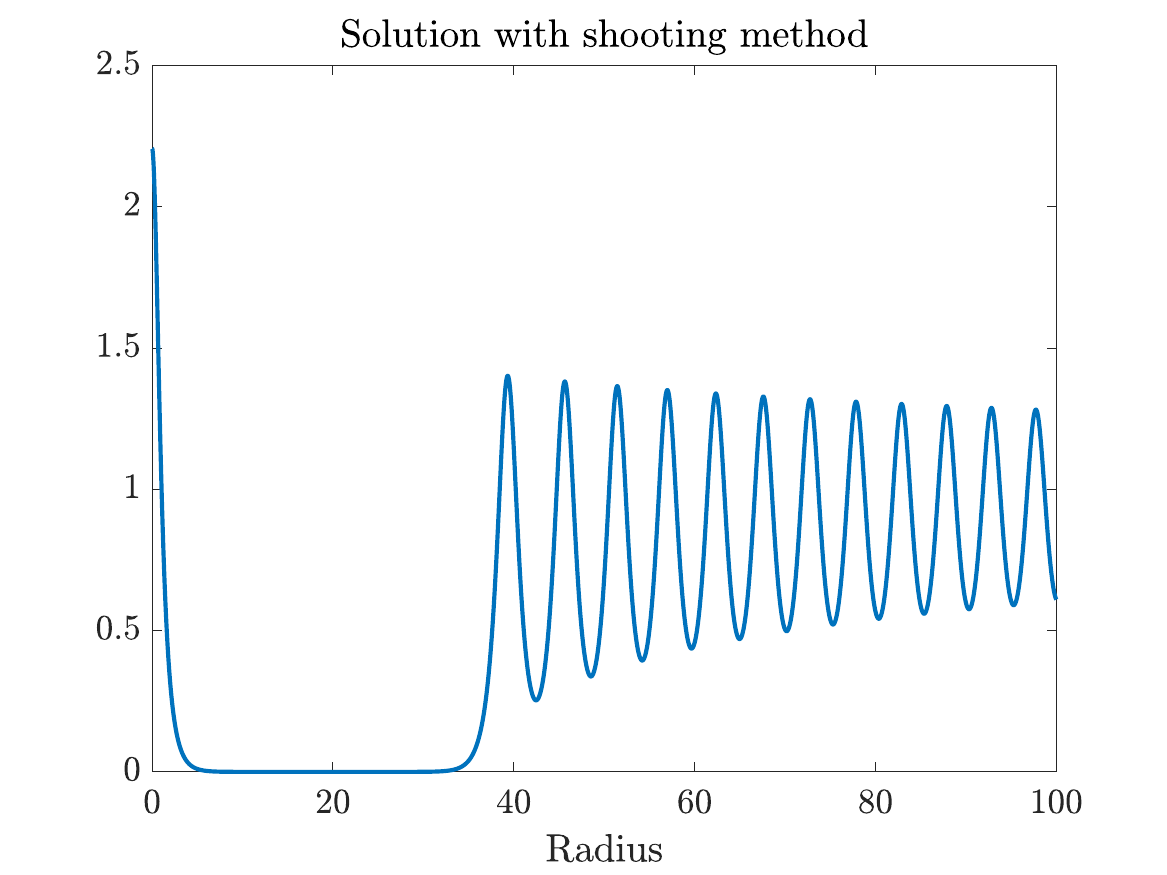}
    \caption{Solution (shooting method)}
 \end{subfigure}
 \hfill
 \begin{subfigure}[b]{0.45\textwidth}
         \centering
    \includegraphics[width=1\textwidth]{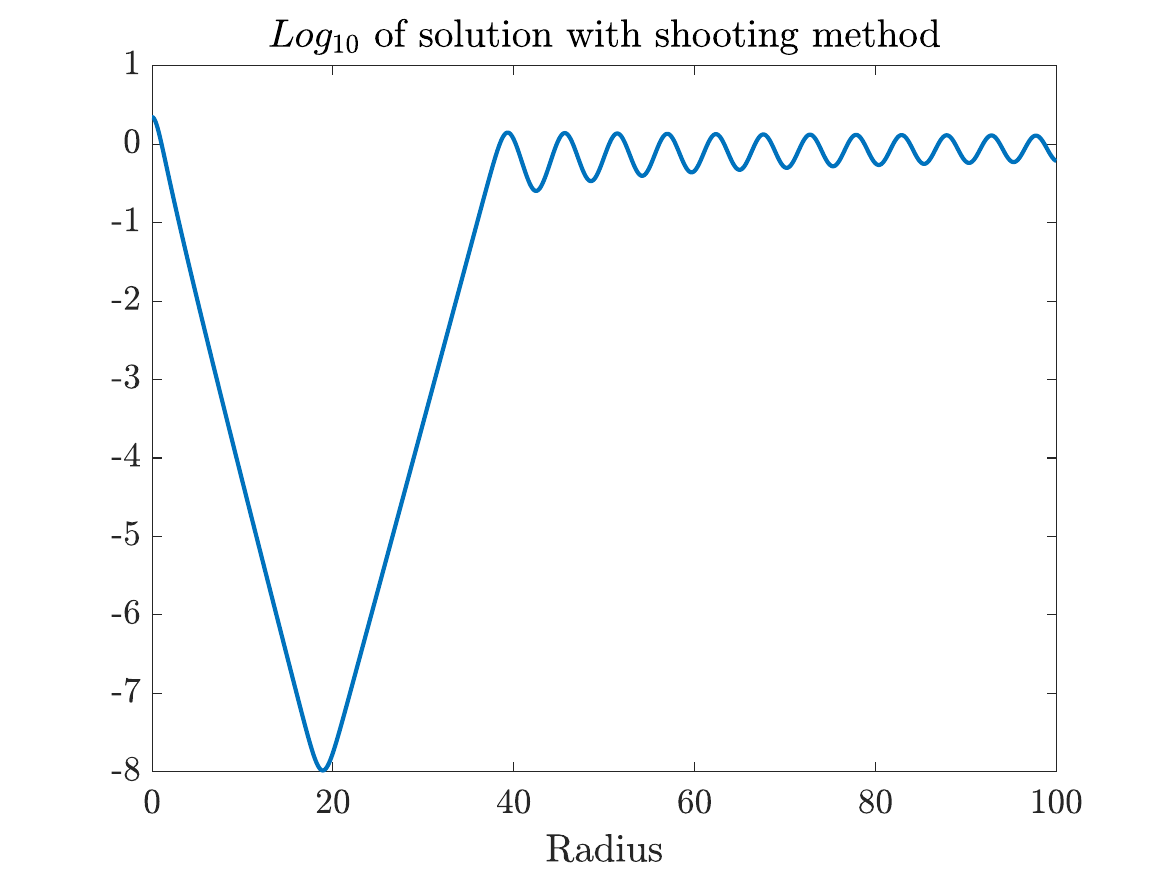}
    \caption{$Log_{10}$ of solution (shooting method)}
 \end{subfigure}

  \begin{subfigure}[b]{0.45\textwidth}
         \centering
    \includegraphics[width=1\textwidth]{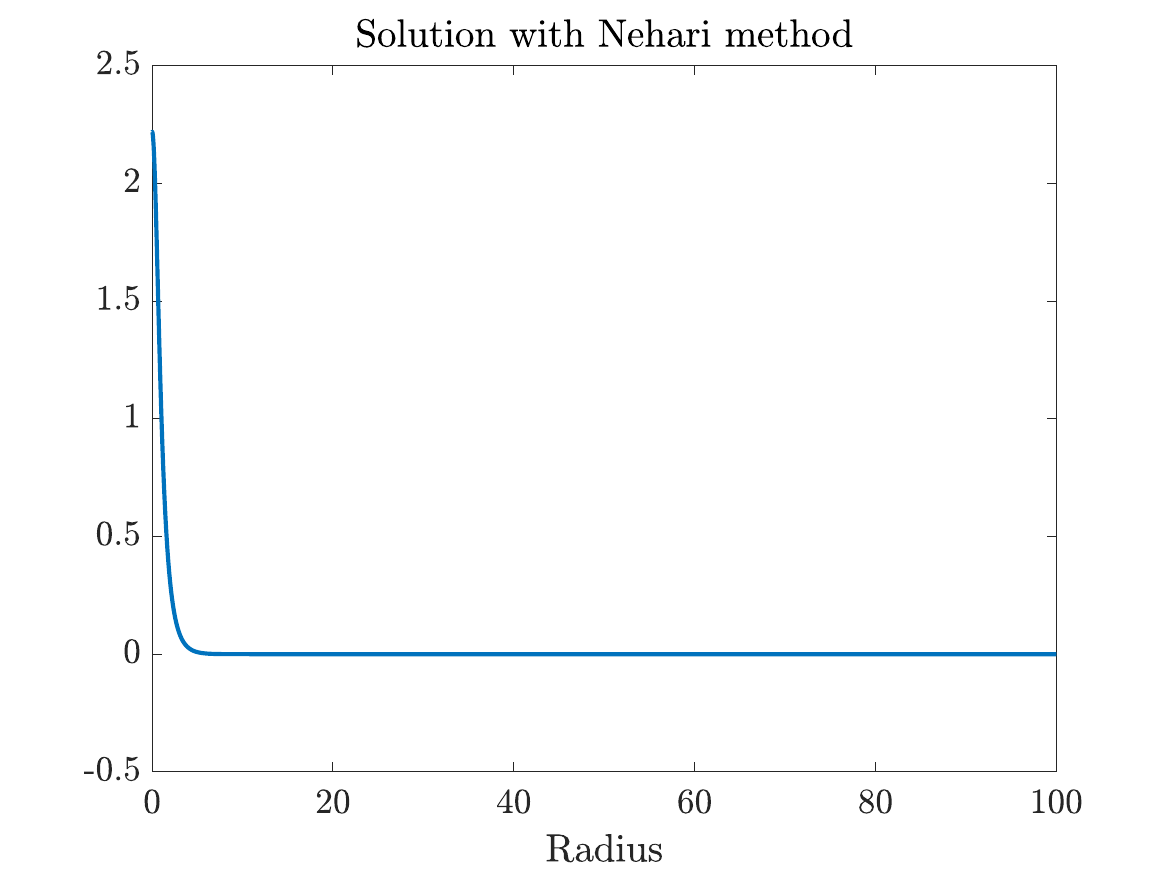}
    \caption{Solution (Nehari method)}
 \end{subfigure}
 \hfill
 \begin{subfigure}[b]{0.45\textwidth}
         \centering
    \includegraphics[width=1\textwidth]{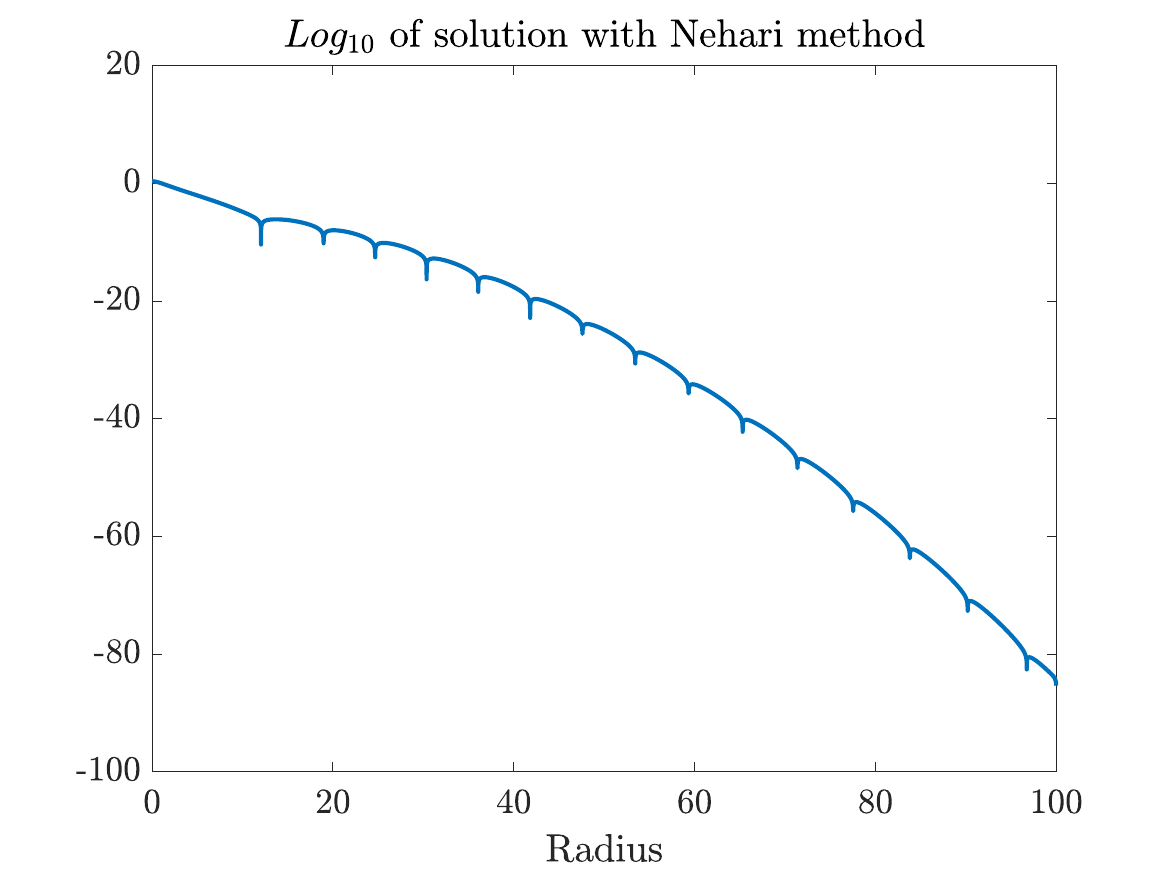}
    \caption{$Log_{10}$ of solution (Nehari method)}
 \end{subfigure}
    \caption{Computation of a ground state on large domain}
\label{fig:large-space-error}
\end{figure}

\begin{figure}[htpb!]
    \includegraphics[width=0.7\textwidth]{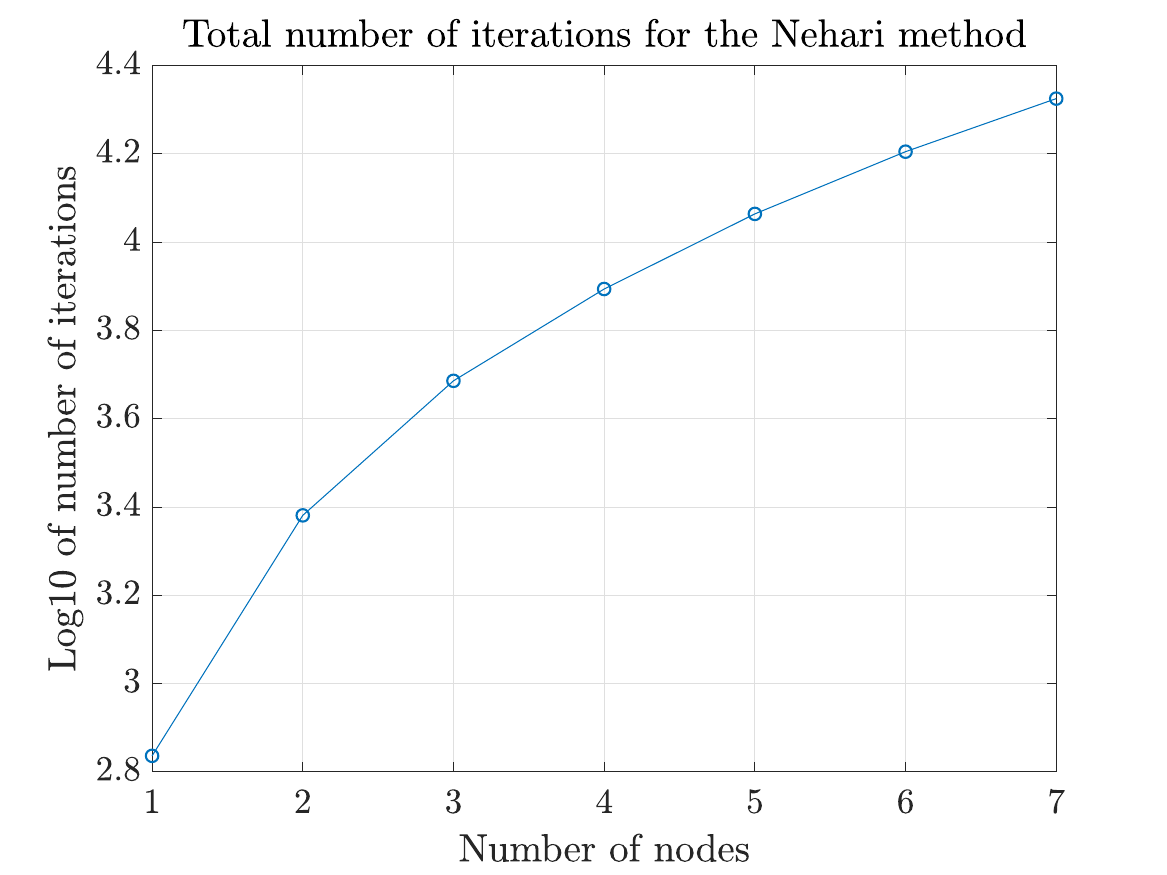}
    \caption{Total number of iterations for the Nehari method depending on the number of nodes}\label{fig:totalnumit}
\end{figure}

We now turn to the number of iterations required to compute a bound state. In the case of the shooting method, this number is naturally bounded by the maximal number of iterations used in the bisection. With a maximal precision set to be $\varepsilon = 10^{-16}$ and an initial interval of length $100$ for the initial data, the number of iterations will always be lower than $18\log_2(10) \approx 60$. The Nehari method does not benefit from such bound on its number of iterations. In Figure \ref{fig:totalnumit}, we depict the number of iterations necessary to the computation of bound states with respect to their number of nodes for $d = 2$, $p = 3$, $R = 30$ and $N = 2^{12}$. In each case, we use the following initial data
\begin{equation*}
    u_0(r) = \cos(r)e^{-r^2/30}.
\end{equation*}
We remark that this initial data has a large number of nodes and decreases rapidly to zero. By construction, the algorithm selects the desired number of nodes and the excess nodes are discarded. 
We can see that the number of iterations grows rapidly, making the Nehari method numerically costly compared to the shooting method.



Finally, we investigate the convergence properties of these methods with respect to the number of discretization points. To do so, we compute bounds states in dimension $d=2$ and for $p=3$ on the interval $[0,R]$, for $R = 30$, for different numbers of nodes with each method. The number of discretization points is set to be $2^N$ with $N\in\{8,9,10,11,12\}$ and we compute the errors
\begin{align*}
&\quad e_N^{(1)} : = \|\boldsymbol{u}^{(N)} - \boldsymbol{u}^{(\textrm{ref})}\|_{L^1} = h \sum_{k = 1}^{2^N} |\boldsymbol{u}^{(N)}_k - \boldsymbol{u}^{(\textrm{ref})}_k| 
\\ &\mbox{and}\quad e_N^{(\infty)} : = \|\boldsymbol{u}^{(N)} - \boldsymbol{u}^{(\textrm{ref})}\|_{L^\infty} = \sup_{1\leq k\leq 2^N} |\boldsymbol{u}^{(N)}_k - \boldsymbol{u}^{(\textrm{ref})}_k|,
\end{align*}
for each $N$, where $\boldsymbol{u}^{(\textrm{ref})}$ is the bound state computed with $2^{15}$ discretization points. The results are depicted in Figure \ref{ref:convNehari} for the Nehari method and Figure \ref{ref:convShooting} for the shooting method. We can see that the order of convergence of the Nehari method depends on the number of nodes and does not seems to be a specific value. However, we can affirm that it is of order above $1$. In the case of the shooting method, the conclusion is more straightforward since the order of convergence is clearly $1$ regardless of the number of nodes. This is explained by the fact that the positions of the nodes are computed with an error of the order of the space discretization.

\begin{figure}
 \centering
     \begin{subfigure}[b]{0.3\textwidth}
         \centering
         \includegraphics[width=\textwidth]{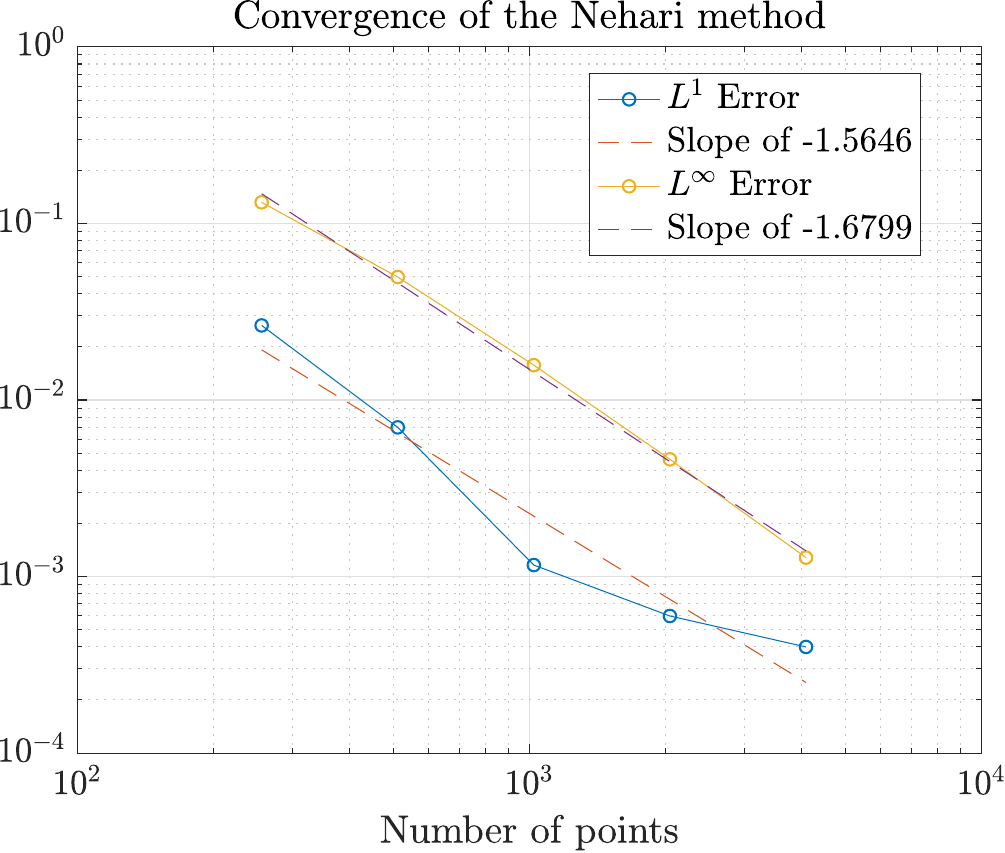}
         \caption{$N_{\textrm{nodes}} = 1$}
     \end{subfigure}
     \hfill
     \begin{subfigure}[b]{0.3\textwidth}
         \centering
         \includegraphics[width=\textwidth]{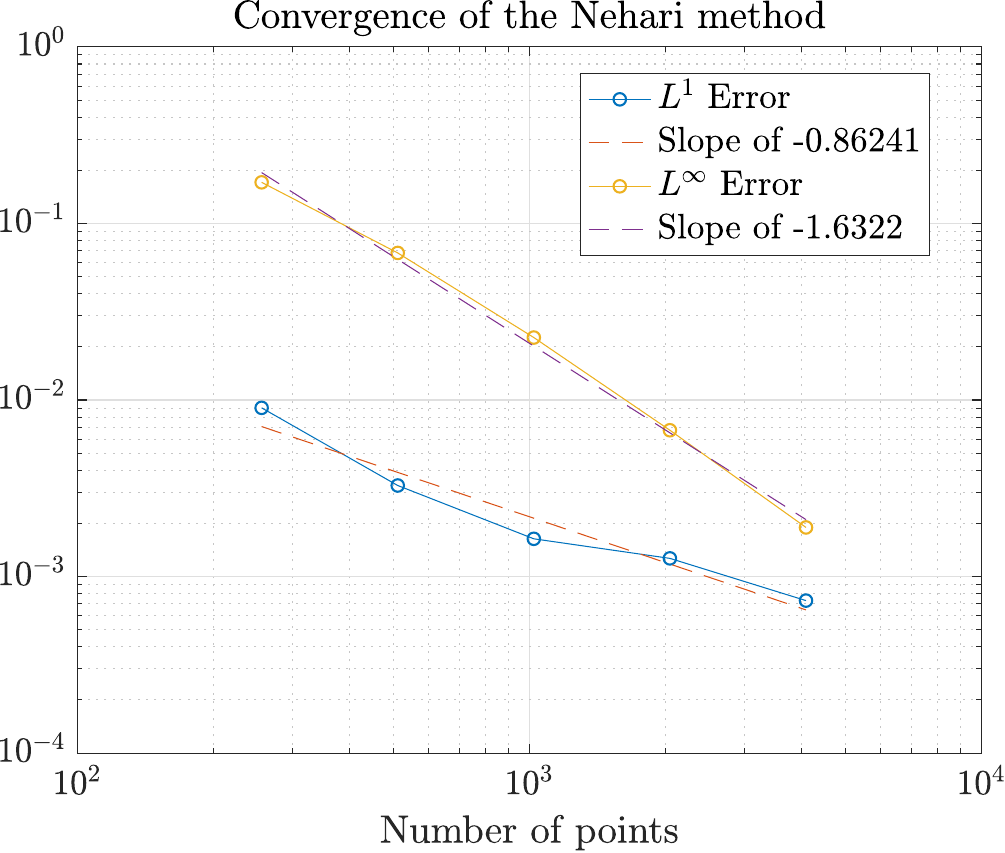}
         \caption{$N_{\textrm{nodes}} = 2$}
     \end{subfigure}
     \hfill
     \hspace{1em}\begin{subfigure}[b]{0.3\textwidth}
         \centering
         \includegraphics[width=\textwidth]{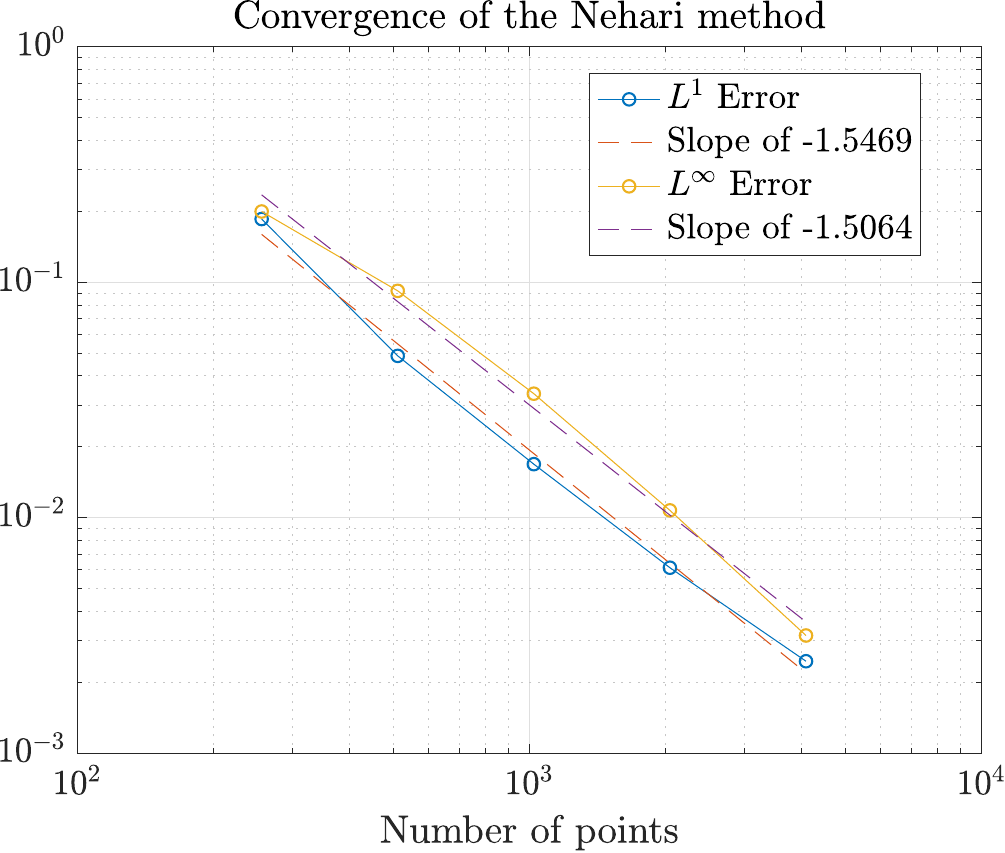}
         \caption{$N_{\textrm{nodes}} = 5$}
     \end{subfigure}

        \caption{Convergence for the Nehari method depending on the number of nodes}\label{ref:convNehari}
\end{figure}

\begin{figure}
 \centering
     \begin{subfigure}[b]{0.3\textwidth}
         \centering
         \includegraphics[width=\textwidth]{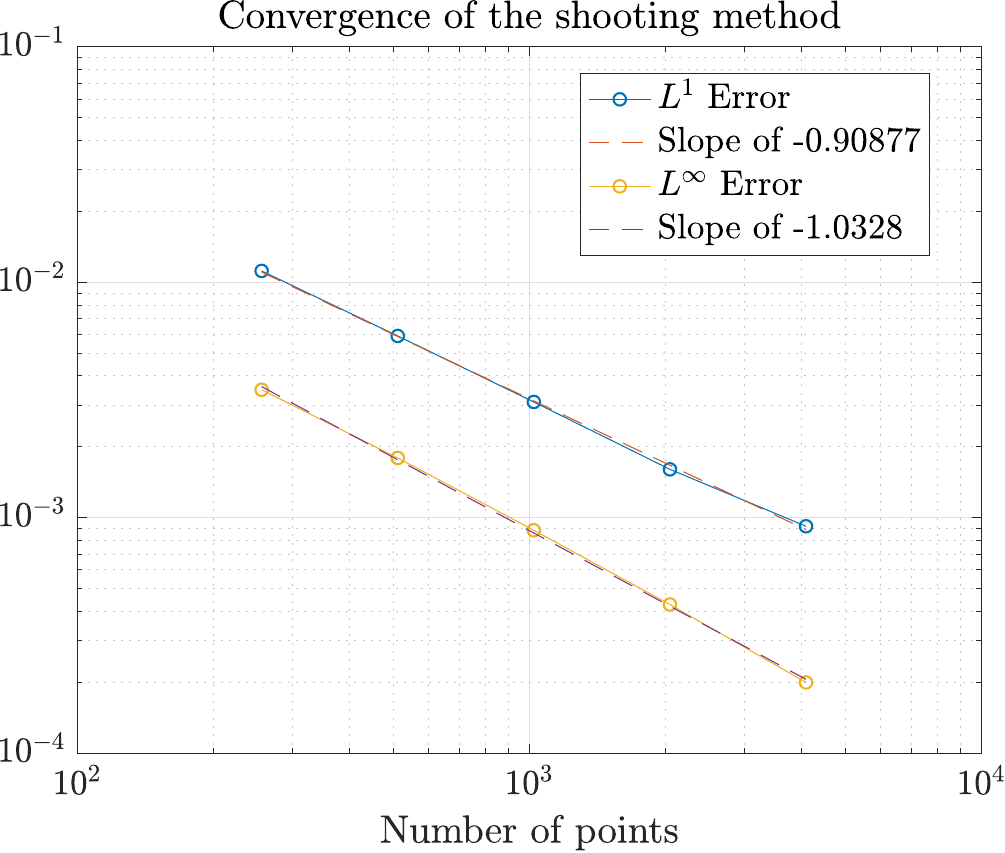}
         \caption{$N_{\textrm{nodes}} = 1$}
     \end{subfigure}
     \hfill
     \begin{subfigure}[b]{0.3\textwidth}
         \centering
         \includegraphics[width=\textwidth]{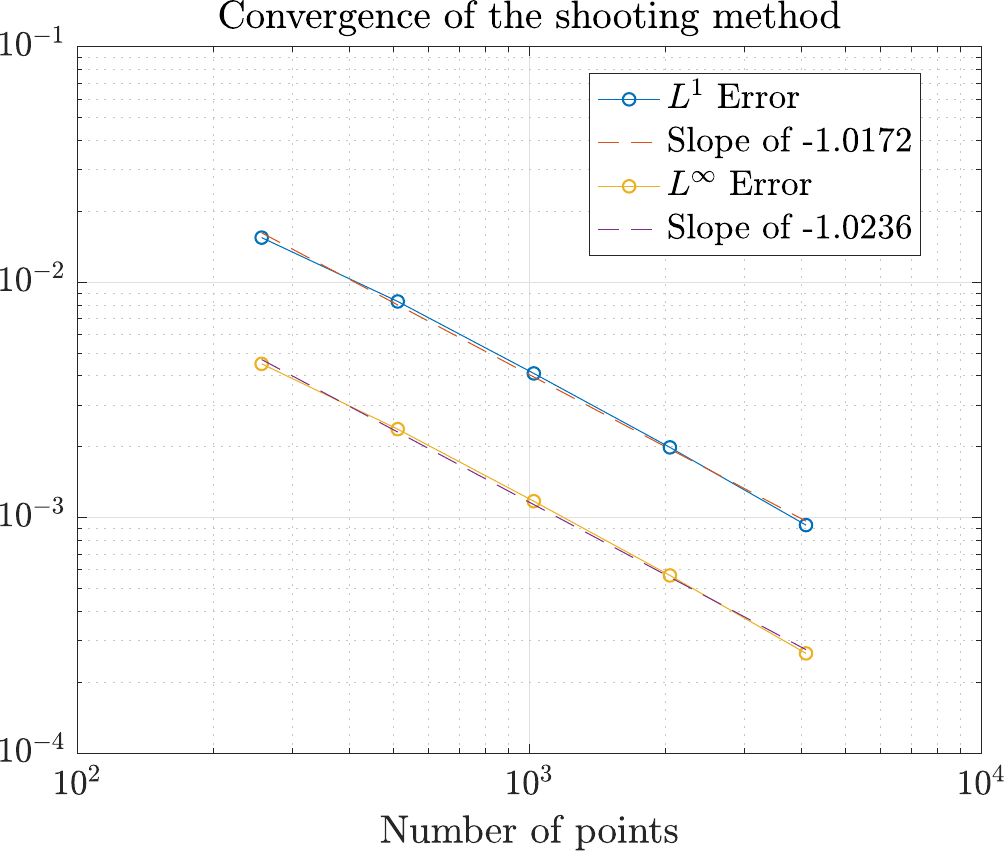}
         \caption{$N_{\textrm{nodes}} = 2$}
     \end{subfigure}
     \hfill
     \hspace{1em}\begin{subfigure}[b]{0.3\textwidth}
         \centering
         \includegraphics[width=\textwidth]{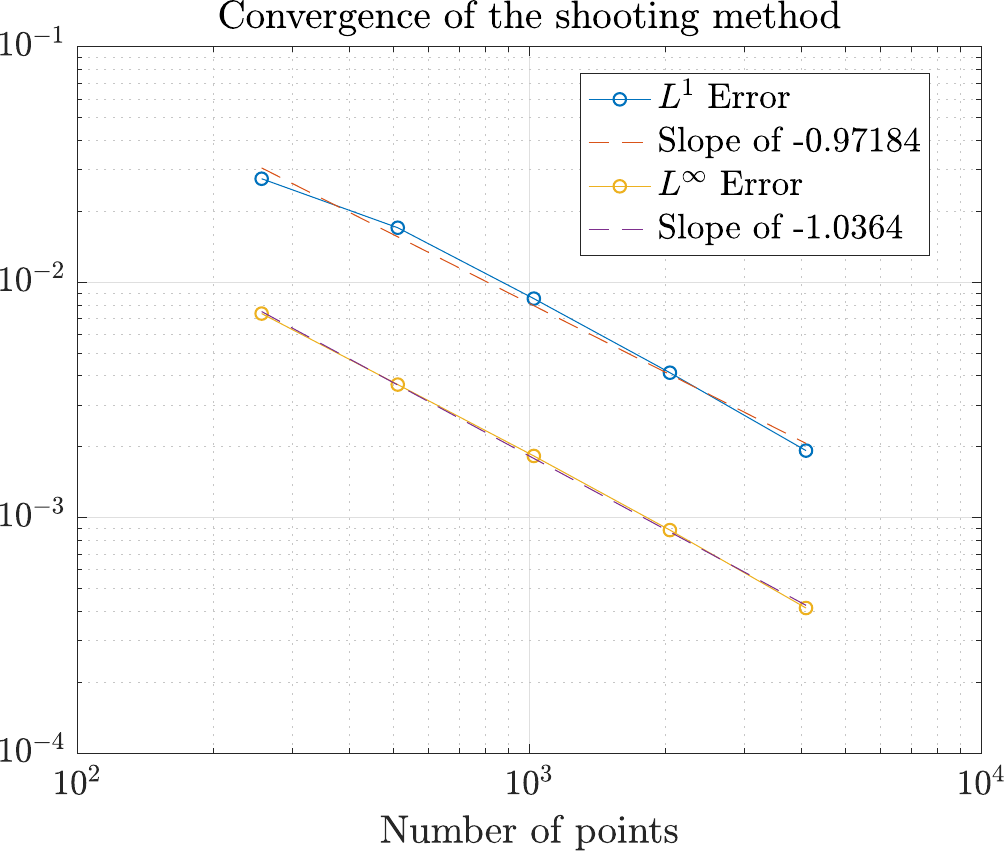}
         \caption{$N_{\textrm{nodes}} = 5$}
     \end{subfigure}

        \caption{Convergence for the shooting method depending on the number of nodes}\label{ref:convShooting}
\end{figure}

We then perform a comparison between the bound state obtained by each method in the same configuration. That is, we compute the errors
\begin{align*}
&\quad E_N^{(1)} : = \|\boldsymbol{u}^{(N,\textrm{Nehari})} - \boldsymbol{u}^{(N,\textrm{Shooting})}\|_{L^1} = h \sum_{k = 1}^{2^N} |\boldsymbol{u}^{(N,\textrm{Nehari})}_k - \boldsymbol{u}^{(N,\textrm{Shooting})}_k| 
\\ &\mbox{and}\quad E_N^{(\infty)} : = \|\boldsymbol{u}^{(N,\textrm{Nehari})} - \boldsymbol{u}^{(N,\textrm{Shooting})}\|_{L^\infty} = \sup_{1\leq k\leq 2^N} |\boldsymbol{u}^{(N,\textrm{Nehari})}_k - \boldsymbol{u}^{(N,\textrm{Shooting})}_k|,
\end{align*}
for each $N\in\{8,9,10,11,12\}$, where $\boldsymbol{u}^{(N,\textrm{Nehari})}$ (resp. $\boldsymbol{u}^{(N,\textrm{Shooting})}$) is the bound state obtained by the Nehari method (resp. the shooting method). The results can be observed in Figure \ref{ref:comparaison} where we can see that, no matter the number of nodes, both methods converge to the same bound state.

\begin{figure}
 \centering
     \begin{subfigure}[b]{0.3\textwidth}
         \centering
         \includegraphics[width=\textwidth]{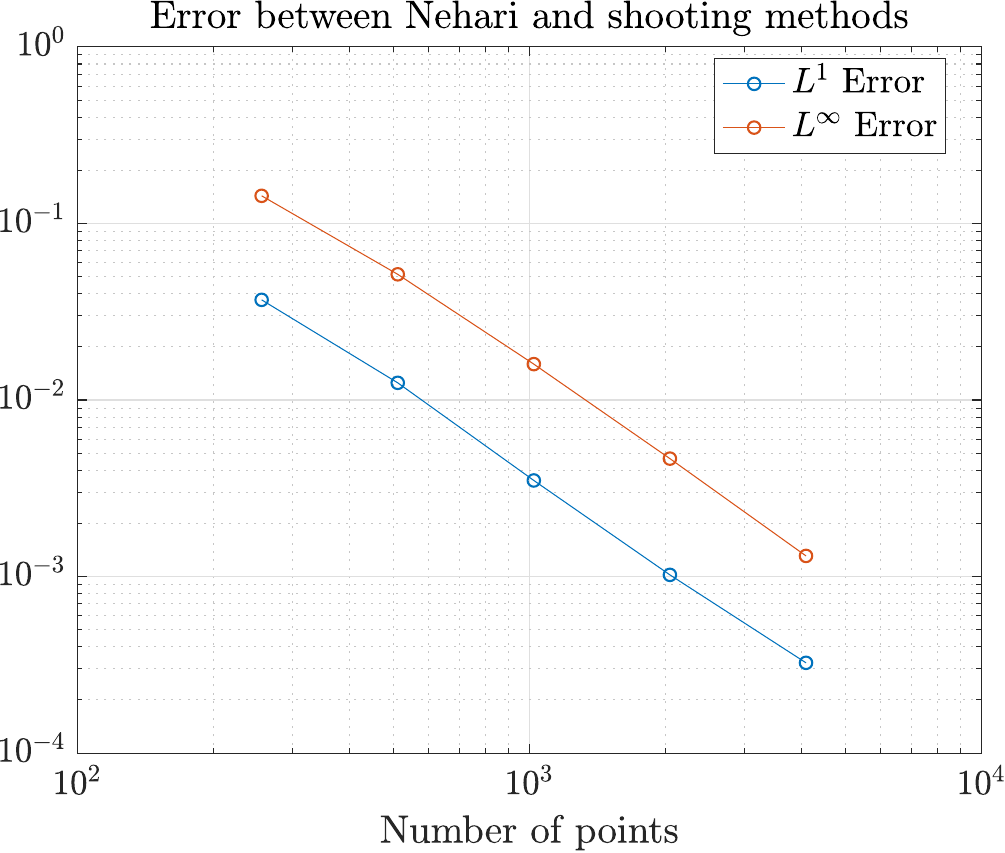}
         \caption{$N_{\textrm{nodes}} = 1$}
     \end{subfigure}
     \hfill
     \begin{subfigure}[b]{0.3\textwidth}
         \centering
         \includegraphics[width=\textwidth]{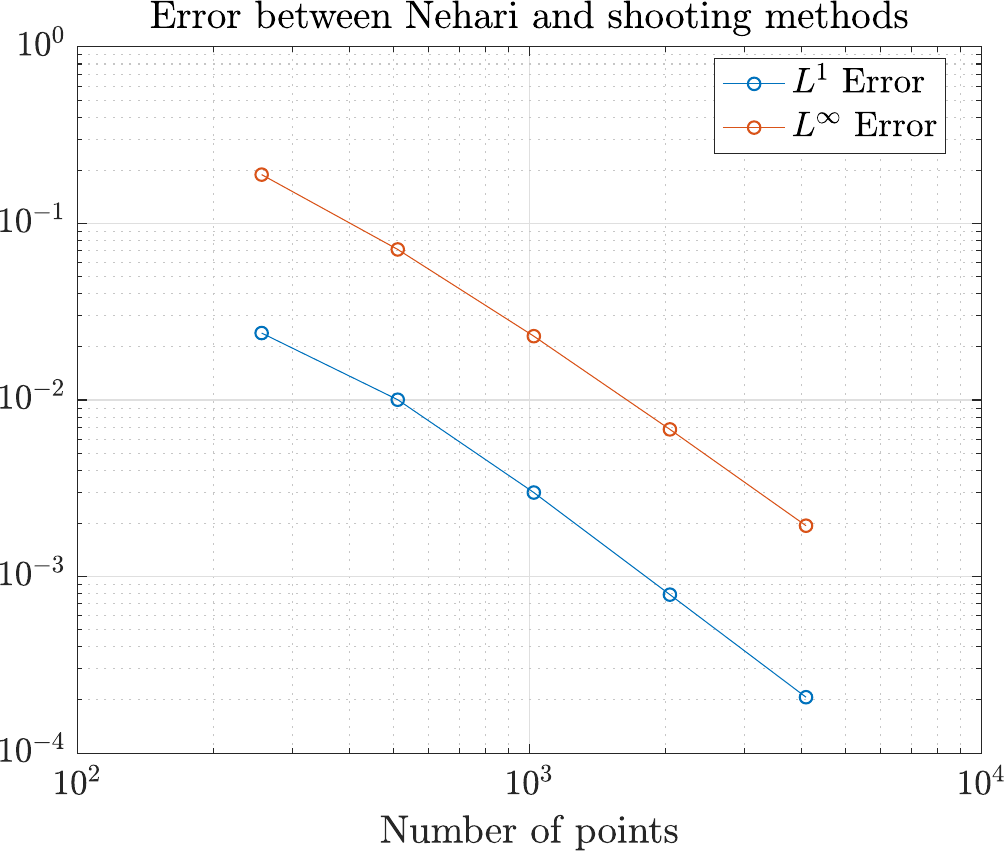}
         \caption{$N_{\textrm{nodes}} = 2$}
     \end{subfigure}
     \hfill
     \hspace{1em}\begin{subfigure}[b]{0.3\textwidth}
         \centering
         \includegraphics[width=\textwidth]{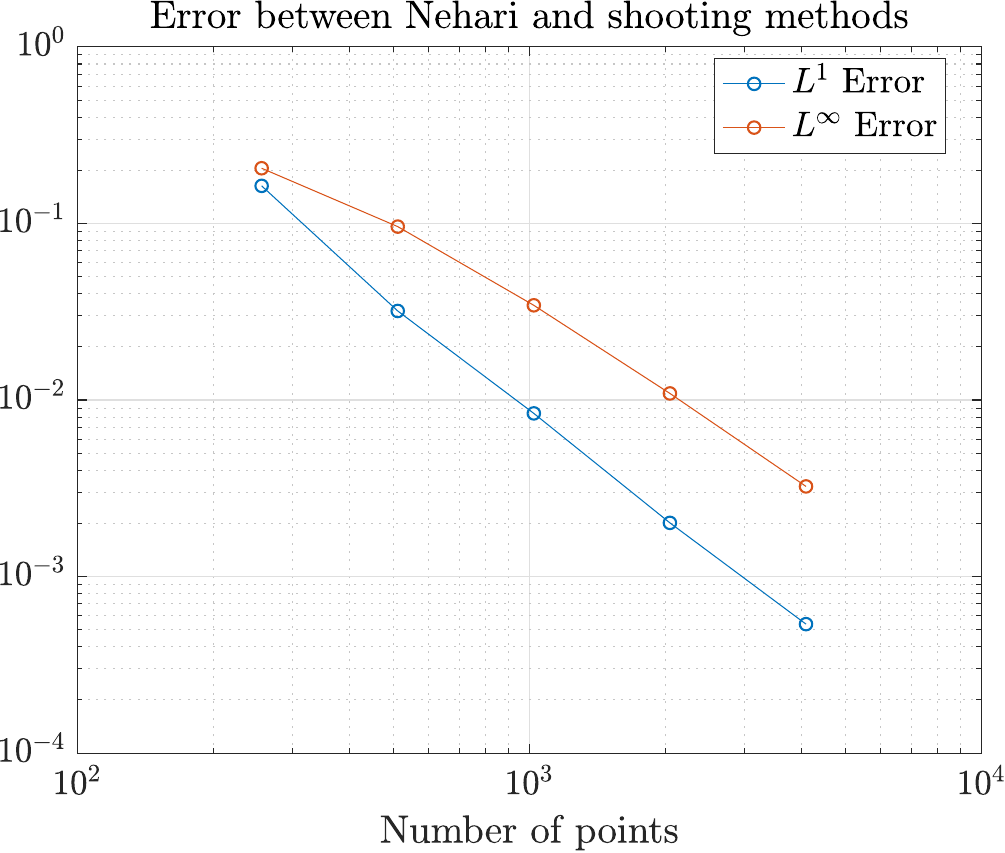}
         \caption{$N_{\textrm{nodes}} = 5$}
     \end{subfigure}

        \caption{Gap between the bound states computed with the Nehari method and the shooting method}\label{ref:comparaison}
\end{figure}

In conclusion, we have studied two numerical methods to compute the bound states of the nonlinear Schrödinger equation in the radial case. The shooting method offers the advantage of being fast but the disadvantage of being less robust, whereas the Nehari method is robust but slow. Based on this observation, this suggest, for numerical experiments, to combine these two methods, that is, to make an initial approximation of the bound state using the shooting method and then refine it using the Nehari method to obtain the desired decay towards zero at infinity.

\section{Numerical Experiments}\label{sec:exp}

In this section, we present some results obtained by numerical experiments consisting in running first the shooting method and then take the outcome as initial data for the Nehari method. 

In Figure \ref{fig:nodsdata}, we consider the case $d=2$ and $p=3$ and depict the relation between the number of nodes $k$ of the bound state $u_k$  and its initial value $u_k(0)$. We fit the data points with a function $k\mapsto a + b\sqrt{k}$ where $a = 0.4841$ (with $95$ percent confidence bounds $[0.4487,0.5194]$) and $b = 2.415$ (with $95$ percent confidence bounds $[2.409,2.422]$).
 
 \begin{figure}[htpb!]
\centering
\includegraphics[width=0.65\textwidth]{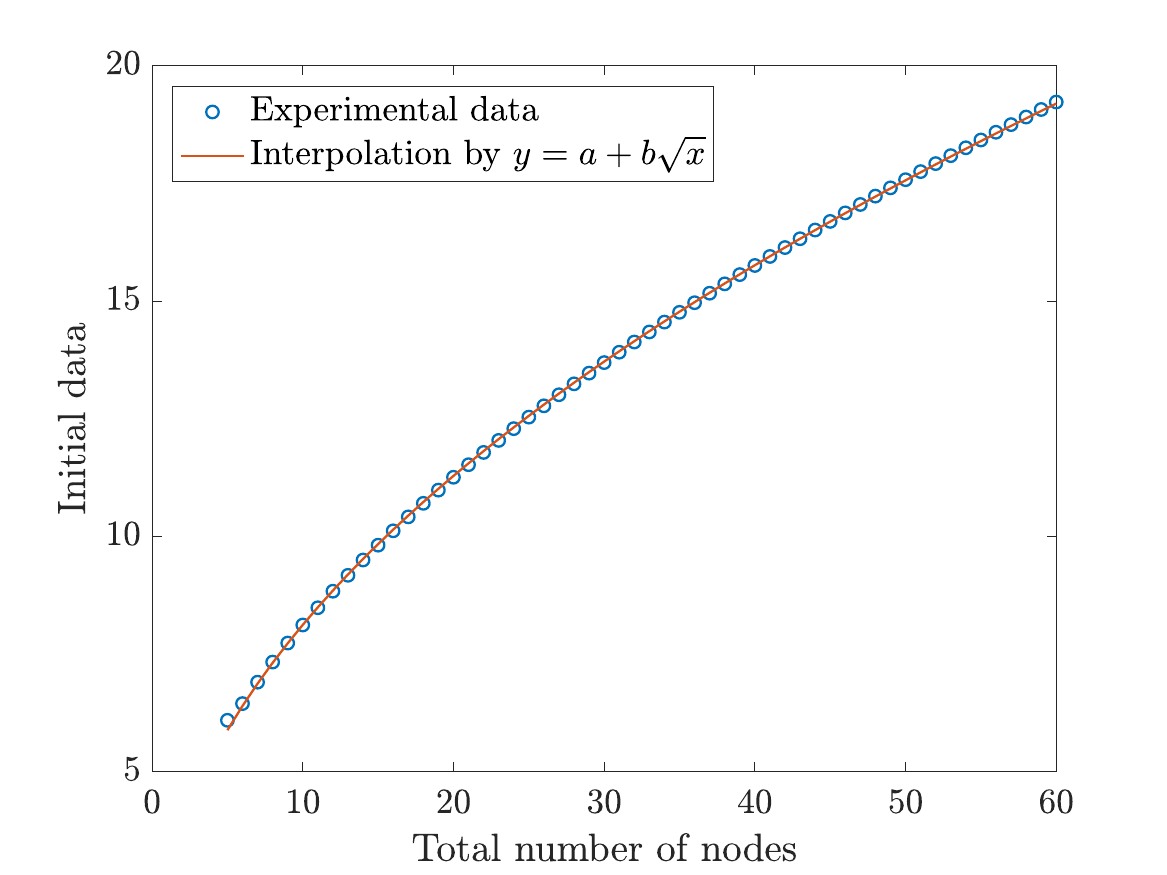}
\caption{Evolution of the number of nodes depending on the value of the initial data}
\label{fig:nodsdata}
\end{figure}

We also studied the positions of the nodes depending on the bound state (that is depending on its total number of nodes). These positions are depicted in Figure \ref{fig:nodesposA}. For each node, the position seems to follow a certain behavior that can be modeled by the function $k\mapsto 1/\sqrt{ak+b}$. We illustrate the value of the coefficients $a$ and $b$ for each node in Figure \ref{fig:nodesposB}.

\begin{figure}
 \centering
     \begin{subfigure}[b]{0.49\textwidth}
         \centering
         \includegraphics[width=\textwidth]{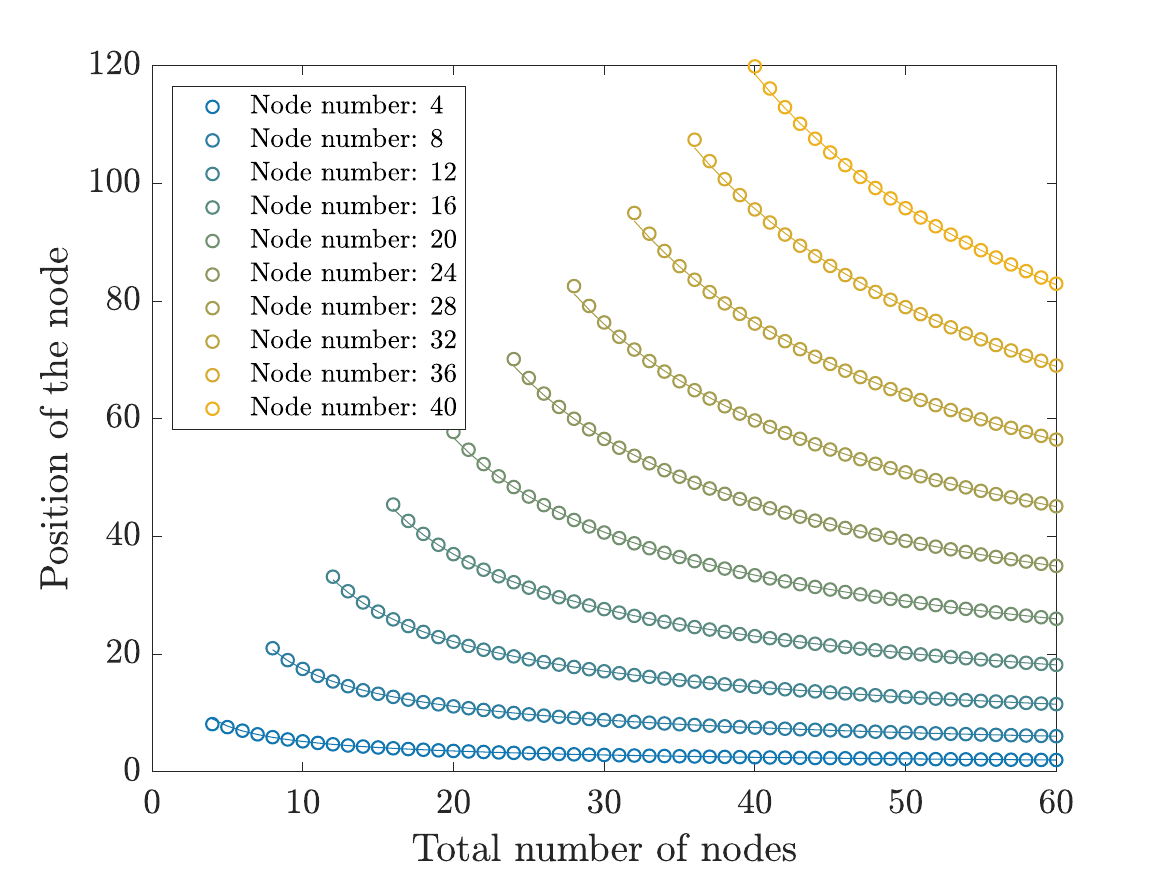}
         \caption{Positions of the nodes}\label{fig:nodesposA}
     \end{subfigure}
     \hfill
     \begin{subfigure}[b]{0.49\textwidth}
         \centering
         \includegraphics[width=\textwidth]{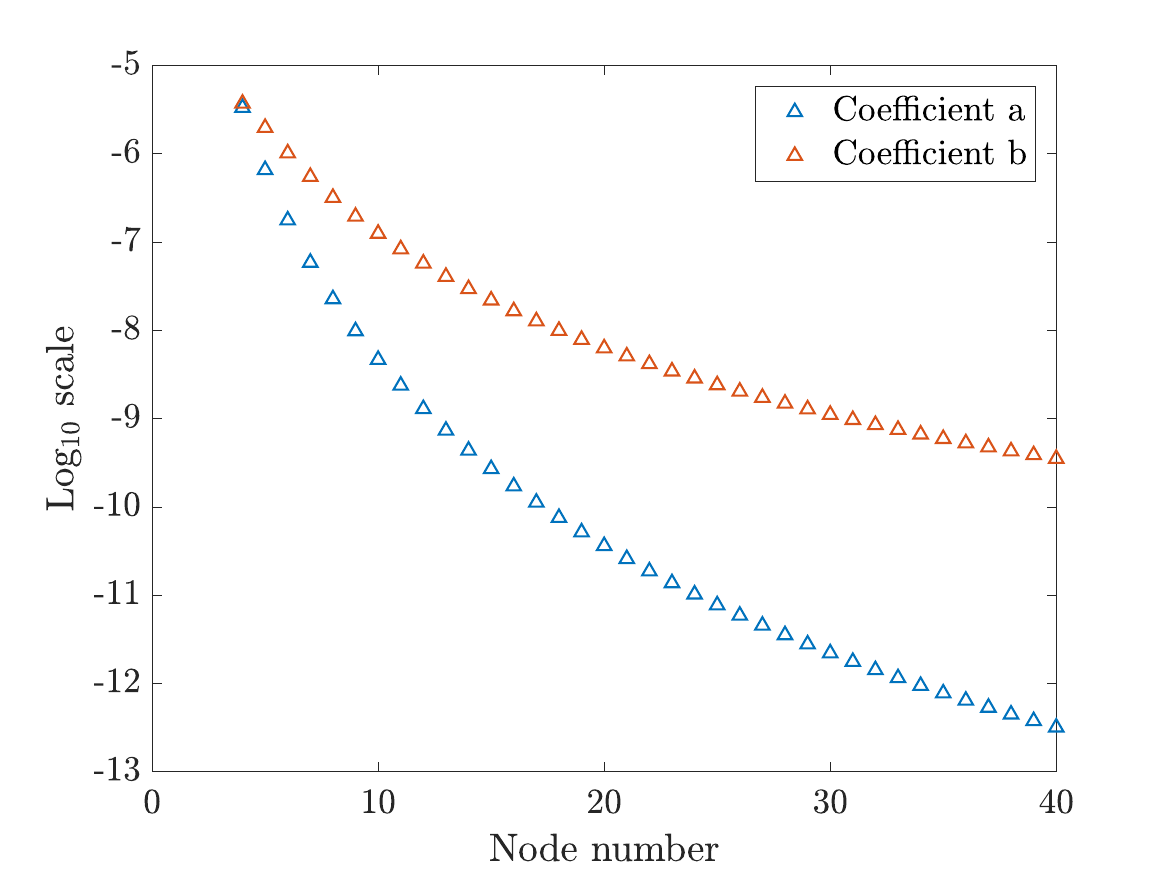}
         \caption{Interpolation's coefficients}\label{fig:nodesposB}
     \end{subfigure}
     \caption{Nodes positions and interpolation by $k \mapsto 1/\sqrt{ak+b}$}\label{fig:nodespos}
\end{figure}

In Figure \ref{fig:maxima}, we plot the positions and (absolute) values of the extrema of the bound states between two consecutive nodes. We observe that for large numbers of nodes in the bound state, the extrema tend to a constant value which is $\sqrt{2}$. This can be explained by the fact that for large $r$ the first derivative term in \eqref{eq:1} vanishes and the solution is close to a soliton of the one dimensional setting whose expression is known to be $r\mapsto\sqrt{2}\,\textrm{sech}(r)$. This is illustrated in Figure \ref{fig:asym} where we superimpose this soliton (adequately shifted) with the (absolute) value of the bound state between its last two nodes (with a total number of nodes equal to $60$).

 \begin{figure}[htpb!]
\centering
\includegraphics[width=0.65\textwidth]{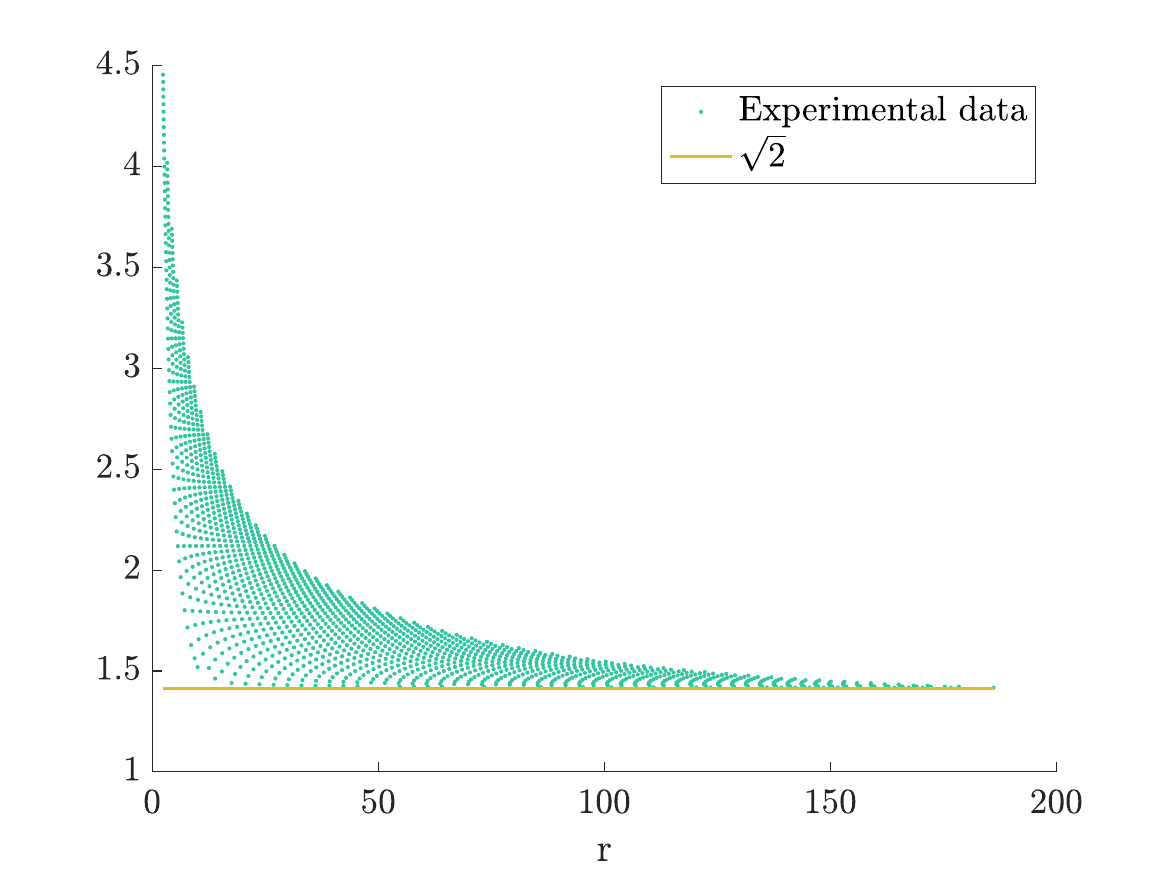}
\caption{Local maxima of the bounds states by positions (x-axis) and values (y-axis)}
\label{fig:maxima}
\end{figure}

\begin{figure}
 \centering
     \begin{subfigure}[b]{0.49\textwidth}
         \centering
         \includegraphics[width=\textwidth]{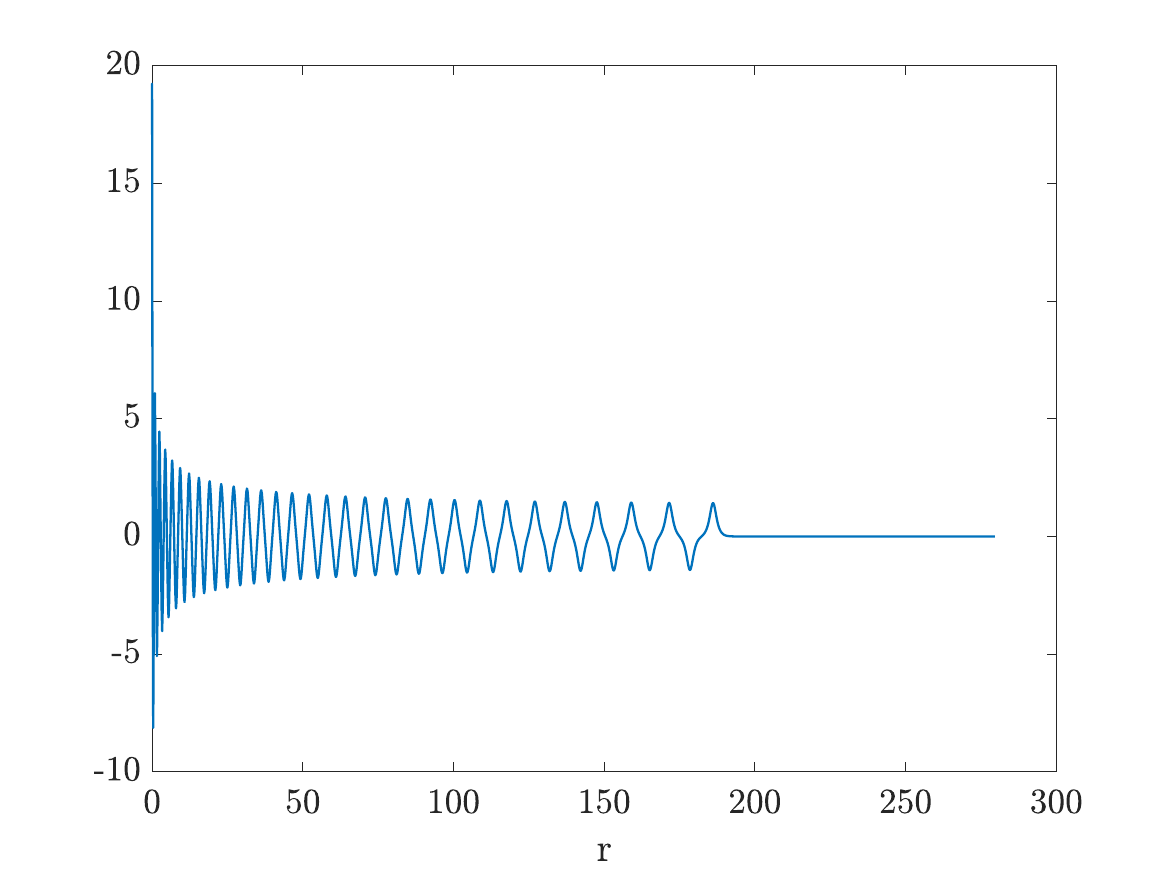}
         \caption{The bound state}
     \end{subfigure}
     \hfill
     \begin{subfigure}[b]{0.49\textwidth}
         \centering
         \includegraphics[width=\textwidth]{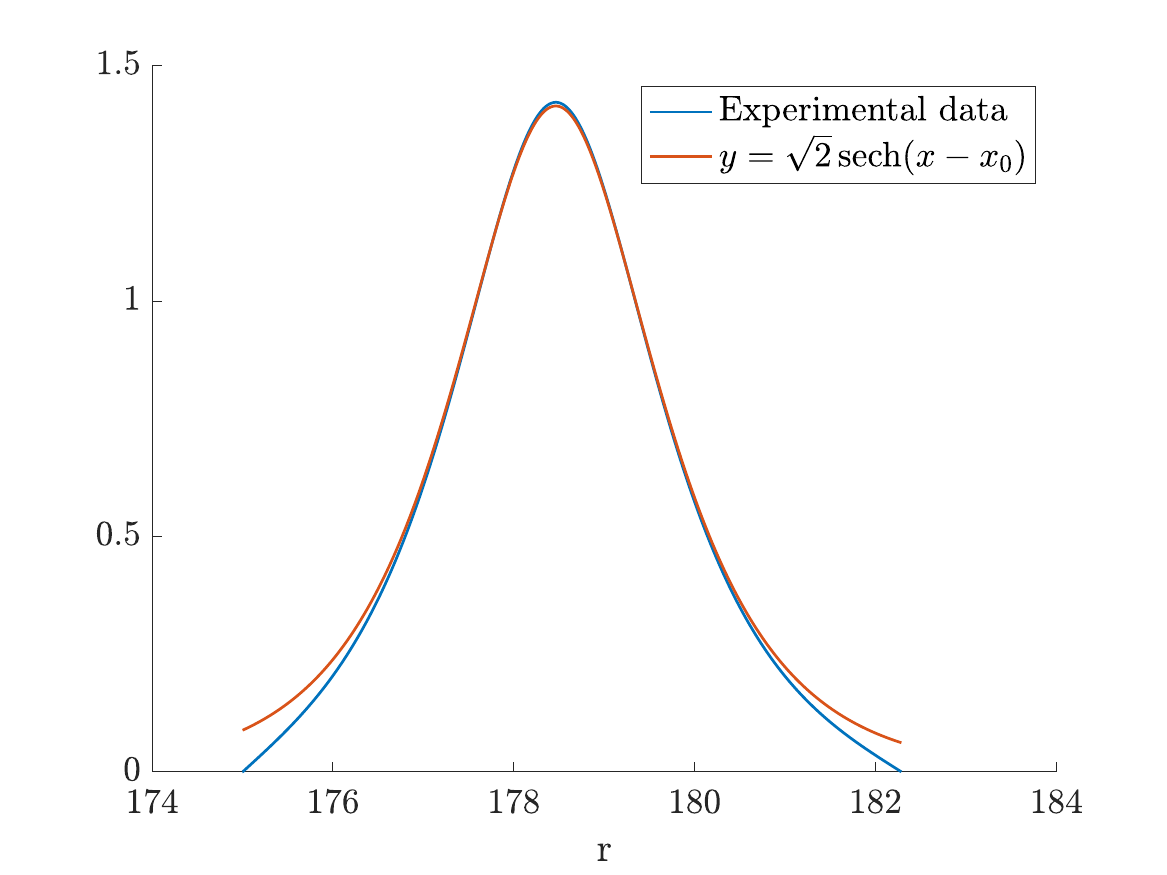}
         \caption{Final oscillation}
     \end{subfigure}
     \caption{Asymptotic behavior of the bound state with $60$ nodes}\label{fig:asym}
\end{figure}

%
%
%
%
%

\appendix

\section{Quantitative Deformation Lemma}
\label{sec:some-techn-results}

In this appendix, we recall the quantitative Deformation Lemma used in the proof of Theorem \ref{thm:least-energy-nodal}.

For $X$ in a Banach space and $S\subset X$, introduce the notation
\[
S_{\delta}:=\{u\in X | \operatorname{dist}(u,S)\leq \delta\},
\]
and for $\varphi:X\to\R$, $c\in \R$, define
\[
\varphi^c=\varphi^{-1}((-\infty,c]).
\]

\begin{lemma}[Quantitative Deformation Lemma \cite{Wi96}]
\label{lem:willem}
  Let $X$ be a Banach space, $\varphi\in\mathcal C^1(X,\R)$, $S\subset
  X$, $c\in\R$, $\eps,\delta>0$ such that 
\[
\norm{\varphi'(u)}_X\geq \frac{8\eps}{\delta}\text{ for all
}u\in\varphi^{-1}([c-2\eps,c+2\eps])\cap S_{2\delta},
\]

Then there exists $\eta\in\mathcal C([0,1]\times X,X)$ such that
\begin{itemize}
\item[(i)] $\eta(t,u)=u$ if $t=0$ or if $u\notin
  \varphi^{-1}([c-2\eps,c+2\eps])\cap S_{2\delta}$,
\item[(ii)] $\eta(1,\varphi^{c+\eps}\cap S)\subset \varphi^{c-\eps}$,
\item[(iii)] $\eta(t,\cdot)$ is an homeomorphism of $X$ for all
  $t\in[0,1]$,
\item [(iv)] $\norm{\eta(t,u)-u}_{X}\leq \delta$ for all $u\in X$ and for all
  $t\in[0,1]$,
\item[(v)] $\varphi(\eta(\cdot,u))$ is non increasing for all $u\in X$,
\item [(vi)] $\varphi(\eta(t,u))<c$ for all $u\in \varphi^c\cap
  S_\delta$ and for all $t\in(0,1]$. 
\end{itemize}
\end{lemma}

\bibliographystyle{abbrv}
\bibliography{biblio}

\end{document}